\documentclass[11pt,reqno]{amsart}

\usepackage[margin=1in]{geometry}

\title[Singular incompressible porous media equation]{On a singular incompressible porous media equation}
\date{\today}
\dedicatory{Dedicated to Professor Peter Constantin on the occasion of his 60th birthday.}

\author{Susan Friedlander}
\address{Department of Mathematics, University of Southern California, Los Angeles, CA
90089}
\email{susanfri@usc.edu}

\author{Francisco Gancedo}
\address{Departamento de An\'alisis Matem\'atico,
Universidad de Sevilla, Sevilla, Spain 41012} \email{fgancedo@us.es}

\author{Weiran Sun}
\address{Department of Mathematics,
The University of Chicago, Chicago, IL 60637}  \email{wrsun@math.uchicago.edu}

\author{Vlad Vicol}
\address{Department of Mathematics,
The University of Chicago, Chicago, IL 60637} \email{vicol@math.uchicago.edu}

\usepackage{amsfonts,amsmath,latexsym,amssymb,verbatim,amsbsy,amsthm}

\usepackage[pdftex]{graphicx}
\usepackage{epstopdf}

\usepackage{times}

\usepackage{color,hyperref}

\theoremstyle{plain}
\newtheorem{theorem}{Theorem}[section]
\newtheorem{definition}[theorem]{Definition}
\newtheorem{lemma}[theorem]{Lemma}
\newtheorem{proposition}[theorem]{Proposition}

\theoremstyle{definition}

\def\tilde{\widetilde}
\numberwithin{equation}{section}

\def\grad{\boldsymbol{\nabla}}

\renewcommand{\hat}{\widehat}

\def\ZZ{{\mathbb Z}}
\def\NN{{\mathbb N}}

\def\RR{{\mathbb R}}
\def\TT{{\mathbb T}}

\def\DD{\mathcal D}

\def\uu{\boldsymbol{u}}
\def\vv{\boldsymbol{v}}
\def\RB{\boldsymbol{R}}
\def\kk{\boldsymbol{k}}
\def\xx{\boldsymbol{x}}
\def\MM{\boldsymbol{M}}
\def\mm{\boldsymbol{m}}

\def\supp{\mathop {\rm supp} \nolimits}

\def\eps{\varepsilon}

\newcommand{\de}{\partial_{\eta}}
\newcommand{\dz}{\partial_{\zeta}}

\begin{document}


\begin{abstract}
This paper considers a family of active scalar equations with transport velocities which are more singular by a
derivative of order $\beta$ than the active scalar. We prove that the equations with $0<\beta\leq 2$ are
Lipschitz ill-posed for regular initial data. On the contrary, when $0<\beta<1$ we show local well-posedness for patch-type weak solutions.
\end{abstract}


%

\maketitle

\section{Introduction}\label{sec:intro}

In this paper we study a singularly modified version of the incompressible porous media equation.
We investigate the implications for the local well-posedness of the equations by modifying, with a fractional derivative, the constitutive relation between the scalar density and the convecting divergence free velocity vector. Our analysis is motivated by recent work~\cite{CCCGW} where it is shown that for the surface quasi-geostrophic equation such a singular modification of the constitutive law for the velocity, quite surprisingly still yields a locally well-posed problem. In contrast, for the singular active scalar equation discussed in this paper, local well-posedness does not hold for smooth solutions,
but it does hold for certain weak solutions.

The incompressible porous media (IPM) equation itself is derived via Darcy's law for the evolution of a flow in a porous medium. In two dimensions it is the active scalar equation for the density field $\rho(\xx,t)$
\begin{align}
& \partial_t \rho+\vv\cdot\grad\rho=0,\label{transportIPM}\\
& \grad \cdot \vv = 0, \label{incom}
\end{align}
where the incompressible velocity field $\vv$ is computed from $\rho$ and the pressure $P$ via Darcy's law~\cite{bear,nield}
\begin{align}
\frac{\mu}{\kappa} \vv=-\grad P-g(0,\rho).\label{darcyI}
\end{align}
Here $\mu$ is the viscosity of the fluid, $\kappa$ is the permeability of the medium, and $g$ is gravity.
For the sake of simplicity, we let $\mu=\kappa=g=1$. The equations are set in either $\RR^{2} \times [0,\infty)$ or $\TT^{2} \times [0,\infty)$. We observe that $\vv$ is determined from $\rho$ by a singular integral operator: using incompressibility~\eqref{incom}, the pressure is obtained from the density as $P = (-\Delta)^{-1}\partial_{x_2}\rho$, which combined with \eqref{darcyI} yields
\begin{align}
\vv=-\grad(-\Delta)^{-1}\partial_{x_2}\rho-(0,\rho) = \RB^{\perp} R_{1} \rho,\label{darcy}
\end{align}
where $\RB = (R_{1},R_{2})$ is the vector of Riesz transforms. There is a considerable body of literature concerning the IPM equation (i.e., \eqref{transportIPM} and \eqref{darcy}), which as well as describing an important physical process, gives a simple model that captures the non-local structure of variable density incompressible fluids~\cite{bear,DPR,nield}.

Active scalar equations that arise in fluid dynamics present a challenging set of problems in PDE. Maybe the best example is the surface quasi-geostrophic equation (SQG), introduced in the mathematical literature in~\cite{CMT}. Here the operator relating the velocity and the active scalar is also of order zero (as is the case for IPM). The SQG equation reads
\begin{align}
& \partial_t \theta +\uu\cdot\grad\theta=0, \label{transportQG}\\
& \uu = \RB^{\perp} \theta \label{uQG}.
\end{align}
For regular initial data, similar results have been proved for IPM and SQG~\cite{CCW,CW,DPR}, while for weak solutions one can find different outcomes~\cite{DDP,Resnick,Shvydkoy}, and for patch-type weak solutions the two systems present completely different behaviors~\cite{DP,G}. The global existence of smooth solutions remains open for both the IPM and SQG equation.

There is a significant difference between the SQG and IPM equations which we explore in this paper: the operator in \eqref{darcy} is {\em even}, while the analogous operator in \eqref{uQG} is {\em odd}. In a recent paper \cite{CCCGW}, the authors investigate what happens in a modified version of the SQG equations, when a (fractional) derivative loss in the map relating the scalar field and the velocity is included, i.e., instead of \eqref{uQG} one has $\uu = \Lambda^{\beta} \RB^{\perp} \theta$. Here $\beta > 0$ and $\Lambda = (-\Delta)^{1/2}$ is the Zygmund operator. It is shown in \cite{CCCGW} that since the divergence-free velocity $\uu$ is obtained from $\theta$ by the Fourier multiplier with symbol $i \kk^\perp |\kk|^{\beta-1}$, which is odd with respect to $\kk$, one obtains a crucial commutator term in the energy estimates. It then follows that the equations are locally well-posed in Sobolev spaces $H^s$ with $s\geq 4$. In this paper we consider a modified version of the IPM equation where the fractional derivative $\Lambda^{\beta}$ is inserted in the constitutive law \eqref{darcy}. More precisely we study the  {singular incompressible porous media} (SIPM) equation, which is given by
\begin{align}
& \partial_t \rho+ \vv \cdot\grad\rho=0,\label{transport}\\
&\vv=-\grad (-\Delta)^{-1}\partial_{x_2}\Lambda^\beta\rho-(0,\Lambda^\beta\rho) = \RB^\perp R_1 \Lambda^\beta \rho,\label{Sdarcy}
\end{align}
where $0< \beta \leq 2$. We observe several significant features of the operator in \eqref{Sdarcy}. It is a pseudo-differential operator of order $\beta$, which is inhomogenous with respect to the coordinates $x_{1}$ and $x_{2}$. Furthermore, it is an {\em even} operator, in the sense that its Fourier multiplier symbol, given explicitly as $- k_{1} \kk^{\perp} |\kk|^{\beta -2}$, is even with respect to the vector $\kk$. These features of the SIPM constitutive law \eqref{Sdarcy} lead to results that are in sharp contrast to those obtained for the singular SQG in \cite{CCCGW}. Considering \eqref{transport}--\eqref{Sdarcy} with $0<\beta \leq 2$ on the spatially periodic domain, we prove that for smooth initial data the equations are locally Lipschitz ill-posed in Sobolev spaces $H^{s}$ with $s >2$. In contrast, when $0<\beta < 1$, we prove local well-posedness for some {\em patch-type} weak solutions of \eqref{transport}--\eqref{Sdarcy}, on the full space. This dichotomy is a reflection of the subtle structure of the constitutive law \eqref{Sdarcy}. The even nature of the symbol relating the active scalar and the drift velocity also proved crucial in showing that for such active scalar equations $L^\infty$ weak solutions are not unique~\cite{DDP,Shvydkoy}.

In Section 2 we prove Lipschitz ill-posedness in Sobolev spaces for the SIPM equation \eqref{transport}--\eqref{Sdarcy} with $0 < \beta \leq 2$,  on $\TT^2 \times [0,\infty)$. The proof follows the lines of a similar result for another active scalar equation where the constitutive law is given by an even unbounded Fourier multiplier, namely the magnetogeostrophic equation (MG) studied in~\cite{FV2,M}. We use the techniques of continued fractions to construct a sequence of eigenfunctions for the operator obtained from linearizing the SIPM equation about a particular steady state. These $C^\infty$ smooth eigenfunctions have real unstable eigenvalues with arbitrarily large magnitudes. Once such eigenvalues are exhibited for the linearized equation the Lipschitz ill-posedness (in the sense that there is no solution semigroup that has Lipschitz dependence on the initial data) of the full nonlinear SIPM equation is proved using classical arguments (see, for example, \cite{FV2,R}). To emphasize that the crucial features of the operator in \eqref{Sdarcy} are that it is even and unbounded, we prove the Lipschitz ill-posedness of a general class of active scalar equations satisfying these properties.

In Section 3 we consider weak solutions of \eqref{transport}--\eqref{Sdarcy} when $0< \beta <1$, on $\RR^2 \times [0,\infty)$. We study solutions evolving from {\em patch-type} initial data
\begin{equation}\label{rhopatcht0}
\rho(\xx,0)=\left\{\begin{array}{cl}
                    \rho^1\quad\mbox{in}&D_0^1=\{\xx\in\RR^2: x_2>f_0(x_1)\}\\
                    \rho^2\quad\mbox{in}&D_0^2=\RR^2\setminus D^1_0,
                 \end{array}\right.
\end{equation}
where $\rho^2 > \rho^1$ are constants, and $f_0(x_1)$ is a smooth function. Patch-type solutions evolving from such initial data are a priori more singular, and they have to be considered in a weak sense (see Definition~\ref{dsd}). In particular, the velocity at points $(x_1,f_0(x_1))$ diverges to infinity. However, the evolution of these weak solutions can be reduced to a contour dynamics equation for the free boundary $f(x_1,t)$, with $f(x_1,0) = f_0(x_1)$. This permits us to follow the construction given in \cite{DPR} for the classical IPM equation ($\beta = 0$). We prove that for initial data of the form \eqref{rhopatcht0} the SIPM equation with $0<\beta<1$ has a unique local in time patch-type solution, with a corresponding smooth interface $f \in L^\infty(0,T; H^s (\RR))$, for $s \geq 4$ and $T>0$.

In the appendix we prove an abstract result concerning continued fractions, which is used in Section~\ref{sec:ill:posed}.

\section{Ill-posedness for active scalars with even unbounded constitutive laws}
\label{sec:ill:posed}

In order to obtain the Lipschitz ill-posedness for the SIPM equation, we first study the equation linearized about a certain steady state. We prove that the associated linear operator has unbounded unstable spectrum, which is the main ingredient in the proof of Theorem~\ref{thm:nonlinear}. In the last subsection we prove that ill-posedness in Sobolev spaces holds for a general class of active scalar equations for which the velocity is obtained from the scalar via an even unbounded Fourier multiplier.

\subsection{Linear ill-posedness in \texorpdfstring{$L^{2}$}{L2} for the SIPM equation for \texorpdfstring{$0< \beta < 2$}{0<beta<2}}
\label{sec:linear}
Consider a density given by
\begin{align*}
\rho(\xx,t)=\overline{\rho}(x_2)
\end{align*}
for a general function $\overline{\rho}:\TT\to\RR$. From \eqref{Sdarcy} we compute the steady state velocity as
\begin{align*}
\bar \vv=(0,\Lambda^\beta\rho)-(0,\Lambda^\beta\rho)=(0,0)
\end{align*}
and therefore we get a steady state of
\eqref{transport}--\eqref{Sdarcy}.

Associated to the steady state $\bar \rho$, one may define the
linear operator $L$ obtained by linearizing the nonlinear term in
\eqref{transport} about this steady state, namely
\begin{align}
  L \rho
  = - \bar \vv \cdot \grad \rho - \vv \cdot \grad \bar \rho
  = - v_2 \partial_2 \bar \rho
  = - R_1^2\Lambda^\beta \rho \; \partial_2 \bar \rho. \label{eq:L:def}
\end{align}
Using the method of continued fractions, see
also~\cite{FSV,FV2}, we shall prove next that the operator $L$ has
a sequence of eigenvalues with positive real part, which diverge to
$\infty$. This in turn implies that the linearized SIPM equation
\begin{align}
  & \partial_t \rho = L \rho \label{linear:SIMP}
\end{align}
is ill-posed from $H^s_x \mapsto L^\infty_t L^2_x$, for any $s\geq
0$. The singular features \eqref{linear:SIMP} shall be used in
Section ~\ref{sec:ill-posedness} to show that the full, nonlinear
SIPM equations are ill-posed in Sobolev spaces, by using a classical
perturbation argument (see, for instance \cite{R}). The following lemma is the key ingredient of the ill-posedness result.

\begin{lemma}\label{lemma:eigenvalue}
Fix an integer $a\geq 1$, let $s\geq 0$, and chose the steady state
$\bar \rho (x_2) = \sin(a x_2)$ of
\eqref{transport}--\eqref{Sdarcy}. For any integer $k \geq 1$, the
linear operator $L$ associated to $\bar \rho$, has a $H^s$
smooth eigenfunction $\rho_k(x_1,x_2)$, with $\Vert \rho_k
\Vert_{H^s}=1$, and corresponding eigenvalue $\lambda_k>0$ which
satisfies
\begin{align}
  \label{eq:lambda:bound}
  \frac{k^\beta}{C_a}
\leq
  \lambda_k
\leq
  \frac{C_a}{2-\beta} k^{1+\beta},
\end{align}
for some constant $C_a \geq 1$, which is independent of $k$.
\end{lemma}
\begin{proof}[Proof of Lemma~\ref{lemma:eigenvalue}]
Following the arguments~\cite{FV2}, we prove the lemma by explicitly
constructing an eigenfunction $\rho$, with associated eigenvalue
$\lambda$, i.e. a solution of
\begin{align}
  L \rho
  = -a \cos(a x_2) R_1^2 \Lambda^\beta \rho
  = \lambda \rho, \label{eigenvalue:problem}
\end{align}
where we make the ansatz that $\rho$ is given explicitly by the Fourier series
\begin{align}
  \label{ansatz}
  \rho(x_1,x_2)
  = \sin(k x_1) \sum_{n\geq 1} c_n \sin(n a x_2) \,,
\end{align}
where $c_n \neq 0$ for all $n \geq 1$. Note that given $\lambda$ and
$c_1$, one may solve \eqref{eigenvalue:problem}--\eqref{ansatz} for
all $c_n$, with $n\geq 2$, but only for suitable values of $\lambda$
do these $c_n$'s converge sufficiently fast to $0$ as $n\rightarrow
\infty$.

Inserting the ansatz \eqref{ansatz} into \eqref{eigenvalue:problem},
and matching the terms with same oscillation frequency, one obtains
that recursion relation
\begin{align}
  &\lambda c_1 + \frac{c_2}{p_2} = 0,
     \mbox{ when } n=1, \label{rec:1}  \\
  &\lambda c_n + \frac{c_{n+1}}{p_{n+1}} + \frac{c_{n-1}}{p_{n-1}} = 0,
     \mbox{ for all } n \geq 2, \label{rec:2}
\end{align}
where for all $n\geq 1$ we have denoted
\begin{align}
  \label{pn}
  p_n = \frac{2(k^2 + n^2 a^2)^{1-\beta/2}}{a k^2}.
\end{align}
Note that $p_n$ grow unboundedly as $n\rightarrow \infty$, whenever
$\beta<2$, and they are monotonically increasing. To solve \eqref{rec:1}--\eqref{rec:2} it is standard to
introduce
$
\eta_n = (c_n p_{n-1})/(c_{n-1} p_n),
$
which solves
\begin{align}
  &\lambda p_1 + \eta_2 = 0,
     \mbox{ when } n=1, \label{eq:rec:1}  \\
  &\lambda p_n + \eta_{n+1} + \frac{1}{\eta_n} = 0,
     \mbox{ for all } n \geq 2. \label{eq:rec:2}
\end{align}
Note that if $\lambda$ is known, the recursion
\eqref{eq:rec:1}--\eqref{eq:rec:2} may be used to determine the
values of $\{\eta_n\}_{n\geq 2}$, and also the sequence $\{c_n\}$ by setting $c_1=p_1$ and
\begin{align}
  c_n = p_n \eta_n \ldots \eta_2 \label{cn:def}
\end{align}
for all $n\geq 2$. The compatibility of
\eqref{eq:rec:1} and \eqref{eq:rec:2} requires that $\lambda$ is
given by a root of the characteristic equation
\begin{align}
  \lambda p_1
  = \frac{1}{ \lambda p_2 - \frac{1}{ \lambda p_3 - \frac{1}{ \lambda p_4 - \dots }}}. \label{therecursion}
\end{align}
To solve the characteristic equation \eqref{therecursion} we appeal to the following abstract result about continuous fractions, whose proof we give in Appendix~\ref{app:fractions} below.
\begin{theorem} \label{thm:*}
Assume that the sequence of real numbers $\{p_n\}_{n\geq 1}$ satisfies
\begin{align}
 0 < p_n < p_{n+1}  \label{p:cond:i}
\end{align}
for any $n \geq 1$, and the $p_n$'s are unbounded, that is
\begin{align}
 \lim_{n\to \infty} p_n = \infty.  \label{p:cond:ii}
\end{align}
 Then there exists a real positive solution $\lambda_*$ of \eqref{therecursion}, such that
 \begin{align}
 \frac{1}{\sqrt{p_1 p_2}} < \lambda_* <  \frac{1}{\sqrt{p_1 p_2 - p_1^2}} \label{eq:lambda*:bounds}
 \end{align}
 holds, and the sequence $\{ c_n \}_{n\geq 2}$ defined by \eqref{cn:def} decays exponentially fast for large enough $n$ and we have the estimate
\begin{align}
\| n^s c_n \|_{\ell^2(\NN)} \leq C \left( |n_0|^{s} + \frac{p_2}{p_1} \right)\; \| c_n \|_{\ell^2(\NN)} \label{eq:Sobolev:Estimate}
\end{align}
where $n_0$ is the largest  integer such that $p_{n_0} \leq 4 p_2$, and $C>0$ is a constant.
\end{theorem}

Since the sequence $\{p_n\}$ defined in \eqref{pn} above is monotonically increasing and unbounded when $0<\beta<2$, we may apply  Theorem~\ref{thm:*}, and obtain the existence of a solution $\lambda_k$ to \eqref{therecursion}, which satisfies
\begin{align}
\frac{1}{\sqrt{p_1 p_2}} < \lambda_k < \frac{1}{\sqrt{p_1 p_2-p_1^2}} \label{eq:lambda:b:bound}.
\end{align}
Note that in addition to the existence of $\lambda_k$, Theorem~\ref{thm:*} also guarantees that the coefficients $\{c_n\}$ constructed via \eqref{cn:def} decay exponentially fast after large enough $n$, so that the function constructed in \eqref{ansatz} is smooth, and in particular lies in $H^s$. Moreover, letting
\begin{align*}
C_{s,k}^2 = \sum_{n\geq 1} n^{2s} c_n^2 < \infty
\end{align*}
 we may divide the $c_n$'s by $C_{s,k}$, and define
\begin{align*}
\rho_k(x_1,x_2) = \frac{\rho(x_1,x_2)}{C_{s,k}}
\end{align*}
which is still a smooth eigenfunction of $L$ with eigenvalue $\lambda_k$, and is normalized to have unit $H^s$ norm. Under this normalization, in view of \eqref{eq:Sobolev:Estimate} we may also estimate the $L^2$ norm of $\rho_k$
\begin{align}
\| \rho_k\|_{L^2} = \frac{1}{C_{s,k}} \| c_n \|_{\ell^2(\NN)} \geq \frac{1}{C ( n_0^{s} + p_2/p_1) } \frac{\| n^s c_n \|_{\ell^2(\NN)}}{C_{s,k}} = \frac{1}{C (n_0^{s}+p_2/p_1)} \geq \frac{1}{C_{a,s} (1 + k^s)}, \label{eq:L2:Hs:bound}
\end{align}
where $C_{a,s}$ is a positive constant which depends only on $a$ and $s$. Above we have used that $n_0$ is the largest number such that $p_{n_0} \leq 4 p_1$, which in view of \eqref{pn} may be computed explicitly ($ n_0 \approx 2 k/a$), and $p_2/p_1$ is uniformly bounded in $k$.

To conclude the proof of the lemma it is only left to verify that \eqref{eq:lambda:bound} holds. Inserting the exact form of $p_1$ and $p_2$ from \eqref{pn} into the estimate \eqref{eq:lambda:b:bound}, yields the existence of a positive constant $C_a$ such that \eqref{eq:lambda:bound} holds, thereby concluding the proof of the lemma.
\end{proof}

\subsection{Linear ill-posedness for the SIPM equation when \texorpdfstring{$\beta = 2$}{beta=2}}
\label{sec:beta=2}

When $\beta = 2$ the argument in Section~\ref{sec:linear} does not apply directly  since the corresponding $p_n$ defined by \eqref{pn} are not growing unboundedly. However the linear ill-posedness still holds, as the equation is even more {\em singular}. In fact the argument given here works for $\beta < 4$. Here we do not  construct an explicit eigenvalue of the linearized operator, but instead give a lower bound for the solution at time $t$.

Again we linearize \eqref{transport}--\eqref{Sdarcy} around
$\bar\rho=\sin (a x_2)$. The associated linearized operator is
$L \rho = -a \cos(a x_2) R_1^2 \Lambda^2 \rho$.
To show the linear instability of $\partial_t \rho = L \rho$,  define another linear operator
\begin{align*}
  \tilde{L}\rho = -a \cos(a x_2) R_1^2 \Lambda \rho
\end{align*}
which corresponds to the operator of  \eqref{eq:L:def} with $\beta=1$.
Let $(\lambda_{k}, \tilde \rho_{k})$ be the corresponding sequence of
eigen-pairs of $\tilde{L}$ as constructed in Section \ref{sec:linear}.
Then there exists $C_a$ such that $k C_a^{-1} \leq \lambda_{k}
\leq C_a k^2$. By the definition of $\tilde \rho_{k}$ in
\eqref{ansatz}, we have $\int_{{\TT}^2} \tilde \rho_{k} = 0$.
We now define
\begin{align*}
  \rho_k
  = \Lambda^{-1} \tilde\rho_{k}
  = \sin(k x_1) \sum_{n\geq 1} \frac{\tilde c_{n}}{n} \sin(n a x_2) \,,
\end{align*}
Then
\begin{align*}
     L \rho_k
  = -a \cos(a x_2) R_1^2 \Lambda \tilde \rho_{k}
  = \lambda_{k} \tilde\rho_{k}
  = \lambda_{k} \Lambda\rho_{k}.
\end{align*}
Therefore, the unique solution of
\begin{align*}
   \partial_t \rho = L \rho \,,
\qquad
  \rho(0, x) = \rho_k(x) \,,
\end{align*}
is
$
\rho(t, x) = e^{t \lambda_{k} \Lambda} \rho_k(x),
$
which in turn shows that
\begin{align*}
 \| \rho (\cdot,t)\|_{L^2} \geq e^{t \lambda_k} \|\rho_k\|_{L^2},
\end{align*}
since $\rho_k$ has zero mean on $\TT^2$. Hence  $\| \rho (\cdot,t)\|_{L^2} \geq e^{t \lambda_k} \|\rho(\cdot,0)\|_{L^2}$, and since $\lambda_k$ can be made arbitrarily large by sending $k \to \infty$, it follows that the linearized equations are ill-posed in the $L^2$ norm (in the sense that there is no continuous semigroup at $t=0$).

\subsection{Nonlinear ill-posedness in \texorpdfstring{$H^{s}$}{Hs} for the SIPM equations}
\label{sec:ill-posedness}
We recall cf.~\cite{FV2} the
definition of Lipschitz local well-posedness:
\begin{definition}\label{def:well}
Let $Y \subset X \subset L^2$ be Banach spaces. The initial value problem
for the SIPM equation \eqref{transport}--\eqref{Sdarcy} is called locally
Lipschitz $(X,Y)$ well-posed, if there exist continuous functions
$T: [0,\infty)^{2} \rightarrow (0,\infty)$ non-increasing (with respect to both variables), and $K: [0,\infty)^{2} \rightarrow (0,\infty)$ non-decreasing, so that for every pair
of initial data $\rho^{(1)}_0, \rho^{(2)}_0 \in Y$ there exist
unique solutions $\rho^{(1)}, \rho^{(2)} \in L^{\infty}(0,T;X)$ of
the initial value problem associated to
\eqref{transport}--\eqref{Sdarcy}, that satisfy
\begin{align}
  \Vert \rho^{(1)}(\cdot,t) -\rho^{(2)}(\cdot,t) \Vert_{X}
\leq
  K \Vert \rho^{(1)}_0 - \rho^{(2)}_0 \Vert_{Y}\label{eq:def:well}
\end{align}
for every $t\in [0,T]$. Here $T = T (\Vert \rho^{(1)}_0 \Vert_{Y},\Vert \rho^{(2)}_0 \Vert_{Y})$ and $K= K(\Vert \rho^{(1)}_0 \Vert_{Y},\Vert \rho^{(2)}_0 \Vert_{Y})$.
\end{definition}
The Banach spaces $X,Y$ considered here are $X = H^r$ and $Y=H^s$, with $r\geq 0$ and $s\geq r+1$. Indeed, if $r \leq s <
r+1$, the Lipschitz $(H^r,H^s)$ well-posedness of first order
equations should in general not even be expected, due to the
derivative loss in the non-linearity. The main theorem of this section is:
\begin{theorem}\label{thm:nonlinear}
The SIPM equations, with $0<\beta < 2$, are locally Lipschitz $(H^r,H^s)$ ill-posed, for any $r > 2$ and $s\geq r+1$, in the sense of
Definition~\ref{def:well} above.
\end{theorem}
The main idea of the proof of Theorem~\ref{thm:nonlinear} is to let
$\rho^{(1)}_0$ be the steady state $\bar \rho(x_2) = \sin(a x_2)$, for
some fixed $a\geq 1$, and  $\rho^{(2)}_0 = \bar \rho + \epsilon
\rho_k$, where $k$ is chosen to depend on the Lipschitz constant $K$
and the time of existence $T$, in such a way that $K < 2
\exp(k^\beta T)$. Letting $\epsilon \rightarrow 0$ it will follow
that the linear equation \eqref{linear:SIMP} should be Lipschitz
$(X,L^2)$ well-posed with the same Lipschitz constant $K$, on
$[0,T)$, which gives rise to a contradiction due to the choice of
$k$. In order to implement this program we need to show uniqueness of solutions to the linearized SIPM equations.

\begin{proposition}\label{prop:unique}
Let $\rho \in L^{\infty}(0,T;L^{2}(\TT^{2}))$ be a solution of the initial value problem
\begin{align}
 \partial_t \rho =   L \rho, \qquad \rho(\cdot,0)=0 \label{eq:Linear}
\end{align}
where as before the linear operator $L$ is defined as $ L\rho = -a \cos(a x_2) R_1^2 \Lambda^\beta \rho$, with $a \in \ZZ$, and $0<\beta \leq 2$.
Then for any $t\in(0,T)$ we have $\rho(\cdot,t)=0$.
\end{proposition}
\begin{proof}[Proof of Proposition~\ref{prop:unique}]
We write $\rho$ in terms of its Fourier series as $\rho(\xx,t) = \sum_{\kk \in \ZZ^2} \hat\rho (\kk,t) \exp (i \kk \cdot \xx)$,
and for each $k_1 \in \ZZ$ define
\begin{align*}
 [\rho(k_1,t)]^2 := \sum_{k_2 \in \ZZ} |\hat \rho(k_1,k_2,t) |^2
\end{align*}
which is finite for each $k_1 \in \ZZ$ and $t\in (0,T)$ by the assumption $\rho \in L^\infty(0,T;L^2)$.
Taking the Fourier transform of \eqref{eq:Linear}, and using $a \cos(a x_2) = a (e^{iax_2} + e^{-iax_2})/2$, we obtain that
\begin{align*}
\partial_t [\rho(k_1,t)]^2 
&\leq a |k_1|^\beta \sum_{k_2 \in \ZZ} \left(|\hat \rho(k_1,k_2-a,t)| + |\hat \rho(k_1,k_2-a,t)| \right) |\hat \rho(k_1,k_2,t)| \leq 2 a |k_1|^2 [\rho(k_1,t)]^2
\end{align*}
for $\beta \leq 2$. The proof of the proposition is concluded since $[\rho(k_1,0)]=0$ for each $k_1 \in \ZZ$, as $\rho(\cdot,0)=0$.
\end{proof}

Having established the uniqueness of solutions to the linearized equation, we now give the proof of the nonlinear ill-posedness result.

\begin{proof}[Proof of Theorem~\ref{thm:nonlinear}]
Fix throughout this proof $a\geq 1$ and $\bar \rho(x_{2}) = \sin (ax_{2})$ a steady state of \eqref{transport}--\eqref{Sdarcy}.
Since $r\geq 2$, by assumption $H^r$ is continuously embedded in $H^\beta$ (for any $\beta \in (0,2]$), and the linear operator
\begin{align*}
L \rho = - R_1^2 \Lambda^\beta \rho\; \partial_2 \bar \rho
\end{align*}
maps $X$ continuously into $L^2$. In addition, the nonlinearity
\begin{align*}
N \rho = \RB^\perp R_1 \Lambda^\beta \rho \cdot \grad \rho
\end{align*}
may be bounded as
\begin{align}
 \| N\rho \|_{L^2}\leq C \|\Lambda^\beta \rho\|_{L^2}  \|\grad \rho\|_{L^\infty} \leq C \| \rho\|_{H^r}^{2} \label{eq:N:bound}
\end{align}
for some constant $C>0$, since in two dimensions $H^{r-1} \subset
L^\infty$ for $r>2$. Assume ad absurdum that the SIPM equations are
locally Lipschitz $(X,Y)$ well-posed in the sense of
Definition~\ref{def:well}.

We let $\rho_{0}^{(1)}(x) = \bar \rho(x_2)$, so that
$\rho^{(1)}(x,t) =\bar\rho(x_2)$ is the unique solution in $H^r$ of
\eqref{transport}--\eqref{Sdarcy} with initial data $\rho_0^{(1)}$.
Denote $\| \bar \rho\|_{H^s}$ for simplicity by $\bar C$. Let $\psi_0
\in H^s$ be a smooth function, to be chosen precisely later, such that
$\|\psi_0\|_{H^s}=1$. For each $\eps \in (0,\bar C]$ we may define
$\rho_0^{(2)}(x) = \rho_0^\eps(x)= \bar \rho(x_2) + \eps \psi_0(x)
\in H^s$, and we denote the unique solution in $H^r$ of
\eqref{transport}--\eqref{Sdarcy} with initial data $\rho_0^{\eps}$
by $\rho^\eps$ (instead of $\rho^{(2,\eps)}$). By
Definition~\ref{def:well} there exits a time $T_\eps = T_\eps( \bar
C , \|\rho_0^\eps\|_{H^s})$ and a Lipschitz constant $K_\eps = K_\eps(
\bar C, \| \rho_0^\eps\|_{H^s})$, such that we have
\begin{align*}
\sup_{[0,T_\eps]} \| \rho^\eps(\cdot,t) - \bar \rho(\cdot) \|_{H^r} \leq K_\eps \| \rho_0^\eps - \bar \rho\|_{H^s} = K_\eps\;  \eps
\end{align*}
since $\| \psi_0\|_{H^s} = 1$. Note that $\| \rho_0^\eps\|_{H^s} \leq \|
\bar \rho \|_{H^s} +\eps \leq 2 \bar C$ for all $\eps \in (0,\bar C]$,
and hence due to our assumptions on the functions $T(\cdot,\cdot)$
and $K(\cdot,\cdot)$, there exists a time of existence $\bar T> 0$
and a Lipschitz constant $\bar K > 0$ such that we have
\begin{align}
\sup_{[0,\bar T]} \| \rho^\eps(\cdot,t) - \bar \rho(\cdot) \|_{H^r} \leq \bar K \eps \label{eq:Lip:bound}
\end{align}
for any $\eps \in (0,\bar C]$. That is, $T_\eps$ and $K_\eps$ may be chosen independently on $\eps$.

In view of the definition of $\rho_0^\eps$, we have that $\psi_0 =
(\rho_0^\eps - \bar \rho)/\eps$, and we may write the solution
$\rho^\eps$ as an ${\mathcal O}(\eps)$ perturbation of $\bar \rho$,
i.e.
\begin{align*}
\psi^\eps = \frac{\rho^\eps - \bar \rho}{\eps}.
\end{align*}
It follows from \eqref{eq:Lip:bound} that $\{ \psi^\eps\}_{\eps}$ is
uniformly bounded in $L^\infty(0,\bar T; H^r)$ by $\bar K$ and
$\psi^\eps$ is a solution of
\begin{align}
\partial_t \psi^\eps  = L \psi_\eps + \eps N \psi^\eps , \qquad \psi^\eps(\cdot,0) = \psi_0. \label{eq:PSI:def}
\end{align}
By \eqref{eq:N:bound} we infer that
\begin{align}
\| N \psi^\eps\|_{L^2} \leq C \|\psi^\eps\|_{H^r}^2 \leq C \bar K^2 \label{eq:N:bound:eps}
\end{align}
on $[0,\bar T]$, and hence since $H^r \subset H^\beta$ we infer from \eqref{eq:PSI:def} that $\{ \partial_t \psi^\eps\}_{\eps}$ is uniformly bounded in $L^\infty(0,\bar T; L^2)$, by $C \bar K^2 + \bar K \| \partial_2 \bar \rho\|_{L^\infty}$. Therefore, by the classical Aubin-Lions compactness lemma we obtain that the weak-$*$ limit $\psi \in L^\infty(0,\bar T;H^r)$ is such that $\psi^\eps \to \psi$ strongly in the $L^2$ norm. But sending $\eps \to 0$ in \eqref{eq:PSI:def}, by using \eqref{eq:N:bound:eps} we obtain that $\psi$ is the unique solution of the initial value problem
\begin{align}
\partial_t \psi = L\psi, \qquad \psi(\cdot,0) = \psi_0,\label{eq:PSI:eq}
\end{align}
and satisfies
\begin{align}
\sup_{[0,\bar T]} \| \psi(\cdot,t)\|_{L^2} \leq \bar K. \label{eq:psi:bound}
\end{align}
Uniqueness follows from  Proposition~\ref{prop:unique} above, since \eqref{eq:PSI:eq} is a linear problem.

The proof of the theorem is now concluded by carefully choosing the initial data $\psi_0 \in H^s$ of \eqref{eq:PSI:eq}, in terms of $\bar T$ and $\bar K$. More precisely, by Lemma~\ref{lemma:eigenvalue}, for any $k \geq 1$ we may find a smooth eigenfunction $\rho_k$ of the operator $L$, normalized to have $H^s$ norm equal to $1$, such that its associated eigenvalue satisfies $\lambda_k \geq k^\beta /C_a$ (where $C_a$ is a positive constant). It follows that the solution $\psi(x,t)$ of \eqref{eq:PSI:eq} with initial condition $\psi_0 = \rho_k$, is given by $\exp(t \lambda_k) \psi_0(x)$ (again we invoke   Proposition~\ref{prop:unique} for uniqueness). Therefore, recalling how $\rho_k$ was constructed,  by \eqref{eq:L2:Hs:bound} we obtain
\begin{align}
\| \psi(\cdot,\bar T)\|_{L^2} = \exp(\bar T \lambda_k) \| \rho_k \|_{L^2} &\geq \frac{\exp(\bar T \lambda_k)}{C_{a,s} k^s} \geq \frac{\exp(\bar T k^\beta/C_a)}{C_{a,s} k^s}
\end{align}
where $C_{a}$ and $C_{a,s}$ are constant that may depend on $a$ and $s$. Since for any given $\bar T, \bar K >0$, we can find a sufficiently large $k$ such that
$\exp(\bar T k^\beta/ C_a)/ (C_{a,s} k^s) \geq 2 \bar K$ the proof is now completed, since we arrived at a contraction with \eqref{eq:psi:bound}.
\end{proof}

\subsection{Ill-posedness of active scalar equations with singular even constitutive law}
\label{sec:general:ill}

The method used to prove ill-posedness for the SIPM equations may be directly
generalized to show the ill-posedness for a class of active
scalar equations of the type
\begin{align}
& \partial_t \theta + \uu \cdot \grad \theta = 0,\label{eq:active-general:1} \\
& \grad \cdot \uu =0, \; \uu = \MM \theta,   \label{eq:active-general:2}
\end{align}
where $(\xx,t) \in \TT^{d} \times [0,\infty)$. The $d$-dimensional vector field $\uu$ is obtained from $\theta$ via the Fourier multiplier operator $\MM$, which is given explicitly in term of the Fourier symbol $\mm = (m_1, \ldots, m_{d-1}, m_d) \colon \ZZ^{d} \to \RR^{d}$. We denote the frequency variable by $\kk$. In this section we give sufficient conditions for $\MM$ which
ensure the ill-posedness of \eqref{eq:active-general:1}--\eqref{eq:active-general:2}.

Let $j \in \{1,\ldots,d\}$ be a fixed coordinate, which for ease of notation we simply take to be $j=d$. We write $\kk'$ to denote the $d-1$ dimensional vector $(k_1,\ldots,k_{d-1}) \in \ZZ^{d-1}$. We assume the following hold:
\begin{enumerate}
\item \label{eq:PM0} $\mm(\boldsymbol{0}',a) = 0$ for a given positive integer $a$ (which we fixed throughout this section); that is, $\bar \theta = \sin(a x_d)$ is a steady state solution of \eqref{eq:active-general:1}--\eqref{eq:active-general:2} with corresponding velocity $\bar \uu= 0$;
\item \label{eq:PM00} $m_d(\kk)$ is a real positive rational function, that is even in $\kk$;
\item \label{eq:PM1} $m_d(\kk',na) \to \infty$ as $|\kk'|\to \infty$, for any fixed $n \in \NN$;
\item \label{eq:PM2} $m_d(\kk',na) \to  0$ as $n \to \infty$, for any fixed $ \kk' \in \ZZ^{d-1}$;
\item \label{eq:PM3} $m_d(\kk',(n+1)a) < m_d(\kk', na)$ for all $n \in \NN$, and any fixed $\kk' \in \ZZ^{d-1}$;
\item \label{eq:PM4} $|\mm(\kk)| \leq C (1 + |\kk|)^{r_{0}}$ for some $r_{0}\geq 0$ and $C>0$, for all $\kk \in \ZZ^{d}$.
\end{enumerate}
Examples of such equations are given by the magneto-geostrophic equation introduced
in~\cite{M} (see also \cite{FV2}), and the singular incompressible porous media equation, both in two and three dimensions.
\begin{theorem}
Assume the Fourier multiplier symbol $\mm$ satisfies properties \eqref{eq:PM0}--\eqref{eq:PM3} above. Then the active scalar equation \eqref{eq:active-general:1}--\eqref{eq:active-general:2} is Lipschitz  $(H^r,H^s)$ ill-posed, for any $r > \max \{2,r_{0}\}$ and $s\geq r+1$, in the sense of
Definition~\ref{def:well}.
\end{theorem}

\begin{proof}
We will only prove the
unboundedness of the spectrum of the linearized operator associated
with \eqref{eq:active-general:1}--\eqref{eq:active-general:2}. The nonlinear Lipschitz ill-posedness follows by arguments verbatim to those in Section~\ref{sec:ill-posedness}. It is here where condition \eqref{eq:PM4} is used: to ensure that for $r>\max\{r_{0},2\}$ the bound analogous to \eqref{eq:N:bound} holds. We omit these details. 

By assumption \eqref{eq:PM0}, $\bar \theta(x_d) = \sin(a x_d)$ is a steady state solution of \eqref{eq:active-general:1}, for some given $a \in \NN$.
The linearized operator around $\bar\theta(x_d)$ has the form
\begin{equation}
  \label{def:L-M}
  L \theta(\xx) = - \MM \theta(\xx) \cdot \grad \bar\theta(x_d) = - M_d \theta(\xx) \; \bar\theta'(x_d) = - (m_d \hat \theta)^{\vee}(\xx) a \cos (a x_d).
\end{equation}
We will construct an appropriate eigen-pair $(\lambda, \theta)$ of
$L$, $i.e.$, a solution of
\begin{equation}
  \label{eq:eigen-L-M}
  L \theta = \lambda \theta \,.
\end{equation}
Fix $\kk' \in \ZZ^{d-1}$, and make the ansatz
\begin{equation}
  \label{ansatz-M}
  \theta (\xx)
  = \prod_{i =1}^{d-1} \sin(k_i x_i)
    \sum_{n\geq 1}
      c_n \sin(na x_d) .
\end{equation}
Inserting \eqref{ansatz-M} into
\eqref{eq:eigen-L-M} and matching the corresponding Fourier modes
gives the recursion relation
\begin{align}
  &\lambda c_1 + \frac{c_2}{p_2} = 0,
     \mbox{ when } n=1, \label{rec:1-M}  \\
  &\lambda c_n + \frac{c_{n+1}}{p_{n+1}} + \frac{c_{n-1}}{p_{n-1}} = 0,
     \mbox{ for all } n \geq 2, \label{rec:2-M}
\end{align}
where for all $n\geq 1$ we have denoted
\begin{align}
  \label{pn-M}
  p_n = \frac{2/a}{m_d(\kk',n a)} \,.
\end{align}
We point out that by \eqref{eq:PM2} and \eqref{eq:PM3} we know that the $p_n$ are monotone increasing, and growing unboundedly. Hence, to solve \eqref{rec:1-M}--\eqref{rec:2-M}, as in Section~\ref{sec:linear} we need to find a positive root of the characteristic
equation
\begin{align}
  \label{therecursion-M}
  \lambda p_1
  = \frac{1}{ \lambda p_2 - \frac{1}{ \lambda p_3 - \frac{1}{ \lambda p_4 - \dots }}} \,.
\end{align}
This root exists in view of Theorem~\ref{thm:*}, since the $p_n$'s are increasing and unbounded. We also obtain that the coefficients $c_n$ decay exponentially fast so that the function $\theta$ constructed in \eqref{ansatz-M} is smooth. In addition, we have a bound for $\lambda$ of the form
\begin{align*}
\frac{1}{\sqrt{p_1 p_2}} < \lambda < \frac{1}{\sqrt{p_1 p_2 - p_1^2}}
\end{align*}
which combined with \eqref{eq:PM1} shows that we can find arbitrarily large  $\lambda$, by simply letting $|\kk'|$ be large enough. This shows that the spectrum of the linearized operator $L$ contains eigenvalues of arbitrary large positive part, and therefore the linearized equation is ill-posed, in the sense that it possesses no semigroup that is continuous at $t=0$ in $L^2$.
\end{proof}

\section{Local well-posedness for weak solutions of patch-type for the SIPM equations}

 In this section we consider solutions for a scalar $\rho(\xx,t)$ given by
\begin{equation}\label{rhopatch}
\rho(\xx,t)=\left\{\begin{array}{cl}
                    \rho^1\quad\mbox{in}&D^1(t)=\{x\in\RR^2: x_2>f(x_1,t)\}\\
                    \rho^2\quad\mbox{in}&D^2(t)=\RR^2\setminus D^1(t),
                 \end{array}\right.
\end{equation} where $\rho^1,\rho^2\geq 0$ are constants, $\rho^1\neq\rho^2$ and the common boundary $\partial D^j(t)$ $j=1,2$ is parameterized as $x_2=f(x_1,t)$. If $\rho(\xx,t)$ satisfies \eqref{rhopatch}, \eqref{transport} and \eqref{Sdarcy} we say that it is a patch-type solution. Then SIPM is understood in the distributional sense and its precise definition is as follows:

\begin{definition}\label{dsd}
Let $T>0$. A function $\rho\in L^\infty(0,T;L^\infty(\RR^2))$ satisfying \eqref{rhopatch}  is a weak solution of \eqref{transport}-\eqref{Sdarcy} if for any test function $\phi\in C^\infty_c([0,T)\times\RR^2)$, the following integral equation holds:
\begin{align}
\int_0^T\int_{\RR^2}\rho (\partial_t\phi +\vv \cdot \grad\phi ) d\xx\,dt+\int_{\RR^2}\rho_0(\xx)\phi(\xx,0)d\xx=0,
\label{eq:weak:def}
\end{align}
and $\vv$ may be computed from $\rho$ by means of \eqref{Sdarcy}, i.e. $\vv = \RB^\perp R_1 \Lambda^\beta \rho$ in the sense of distributions.
\end{definition}
 We will show later that if $\rho$ satisfies \eqref{rhopatch} and $f$ is smooth then
\begin{equation}\label{singvel}
|\vv(\xx,t)|\leq C \frac{|\rho^2-\rho^1|}{|x_2-f(x_1,t)|^\beta}\in L^1_{loc}(\RR^2), \forall\,t\geq 0,
\end{equation}
due to $0<\beta<1$. This implies that the nonlinear term in \eqref{eq:weak:def} is well defined and may be bounded as
\begin{align*}
\int_0^T\int_{\RR^2}|\rho \vv\cdot \grad\phi| d\xx\,dt\leq \|\rho\|_{L^\infty}\|\vv\|_{L^1(\supp \phi)}\|\grad\phi\|_{L^\infty}.
\end{align*}

Here we give the main ingredients to obtain the following contour equation for $f$:
\begin{align}
f_t(\eta,t)&=\frac{\rho^2-\rho^1}{C_\beta}\int_{\RR}\frac{(\eta-\zeta)(\de f(\eta,t)-\de
f(\zeta,t))}{((\eta-\zeta)^2+(f(\eta,t)-f(\zeta,t))^2)^\frac{2+\beta}{2}}d\zeta, \label{ec:1} \\
f(\eta,0)&=f_0(\eta),\label{ec:2}
\end{align}
where $\eta\in\RR$, $C_\beta>0$, and $0<\beta<1$.
Next we obtain local-existence for the system above with $\rho^2>\rho^1$.

  To get the evolution for $f$ we need the velocity at $(\eta,f(\eta,t))$ but only in the normal direction. In fact
\begin{align*}
(\eta,f(\eta,t))_t\cdot (-\de f(\eta,t),1)= \vv(\eta,f(\eta,t),t)\cdot (-\de f(\eta,t),1).
\end{align*}
 We understand above expression for the velocity with the following limit
\begin{align}
\vv(\eta,f(\eta,t),t)\cdot(-\de f(\eta,t),1)
&=\lim_{\epsilon\to 0}\vv(\eta-\epsilon\de f(\eta,t),f(\eta,t)+\epsilon)\cdot(-\de f(\eta - \epsilon \de f(\eta,t),t),1).
\label{eq:v:normal}
\end{align}
We shall now prove that the limit in \eqref{eq:v:normal} is exactly the expression on the right side of \eqref{ec:1}. By \eqref{Sdarcy} we have
 the following relation between $v$ and $\rho$
\begin{equation}\label{vrho}
\vv=\partial_{x_1}\Lambda^{-2+\beta}\grad^{\bot}\rho.
\end{equation}
 For $g$ a regular function, it is a classical fact that
\begin{align}
\partial_{x_1}\Lambda^{-2+\beta}g(x)=-\frac{1}{C_\beta} \int_{\RR^2}\frac{x_1-y_1}{|x-y|^{2+\beta}}g(y)dy
\label{eq:RieszPotential}
\end{align}
where $C_{\beta}=(\pi2^{2-\beta}\Gamma(\frac{2-\beta}{2}))/(\beta \Gamma(\beta/2))$ is a normalization constant.
The identity
\begin{align*}
\grad^{\bot}\rho(x,t)=(\rho^2-\rho^1)(1,\de f(\eta,t))\delta(x_2-f(\eta,t)),
\end{align*}
where $\delta$ stands of the Dirac delta function,
combined with \eqref{vrho} and \eqref{eq:RieszPotential} allows us to write
\begin{equation}\label{velnocurve}
\vv(x,t)=-\frac{\rho^2-\rho^1}{C_\beta}\int_{\RR}\frac{(x_1-\zeta)(1,\de f(\zeta,t))}{|x-(\zeta,f(\zeta,t))|^{2+\beta}}d\zeta.
\end{equation}
Using \eqref{velnocurve} we compute the limit in \eqref{eq:v:normal}
\begin{align*}
& \lim_{\epsilon\to 0}\vv(\eta-\epsilon\de f(\eta,t),f(\eta,t)+\epsilon)\cdot(-\de f(\eta - \epsilon \de f(\eta,t),t),1) \notag\\
& \qquad =\lim_{\epsilon\to 0} -\frac{\rho^2-\rho^1}{C_\beta}\int_{\RR} \frac{(\eta-\zeta-\epsilon \de f(\eta,t) )   ( \de f(\zeta,t) - \de f (\eta - \epsilon \de f(\eta,t),t) ) d\zeta}{((\eta-\zeta-\epsilon \de f(\eta,t))^2+(f(\eta,t)-f(\zeta,t)+\epsilon)^2)^{\frac{2+\beta}{2}}}
\\
& \qquad =  \frac{\rho^2-\rho^1}{C_\beta}  \lim_{\epsilon\to 0} \int_{\RR} L^\epsilon(\eta,\zeta) d\zeta.
\end{align*}

We split the integrand $L^\epsilon$ into $L_1^\epsilon (\eta,\zeta)= L^\epsilon(\zeta,\eta) \chi(|\zeta-\eta|>r)$ and $L_2^\epsilon(\eta,\zeta) = L^\epsilon(\zeta,\eta) \chi(|\zeta-\eta|\leq r)$, where $\chi$ stands for the characteristic function, and $r>0$ is a fixed number. Without loss of generality we may take $\epsilon \leq r/ (2 \| \de f\|_{L^\infty})$, since we send it to $0$ anyway. For such  $\epsilon$ we bound $L_1^\epsilon$ pointwise as
\begin{align}
|L_1^\epsilon(\eta,\zeta)|
&\leq \frac{|\eta-\zeta-\epsilon \de f(\eta,t)| \;   |\de f(\zeta,t) - \de f (\eta - \epsilon \de f(\eta,t),t)|}{|\eta-\zeta-\epsilon \de f(\eta,t)|^{2+\beta}} \chi(|\zeta-\eta|>r) \notag\\
&\leq \frac{C \| \de f\|_{L^\infty}}{ |\eta-\zeta-\epsilon \de f(\eta,t)|^{1+\beta} }  \chi(|\zeta-\eta|>r) \leq \frac{C \| \de f\|_{L^\infty}}{ |\eta-\zeta|^{1+\beta} } \chi(|\zeta-\eta|>r). \label{eq:L1eps:bound}
\end{align}
Since the right side of \eqref{eq:L1eps:bound} lies in $L^1(\RR)$, and is independent of $\epsilon$, from the dominated convergence theorem we obtain that
\begin{align}
\lim_{\epsilon\to 0} \int_{\RR} L_1^\epsilon(\eta, \zeta) d\zeta = \int_{|\eta-\zeta|>r}  \frac{(\eta-\zeta)   ( \de f(\zeta,t) - \de f (\eta,t) ) d\zeta}{((\eta-\zeta)^2+(f(\eta,t)-f(\zeta,t))^2)^{\frac{2+\beta}{2}}}. \label{eq:L1eps:limit}
\end{align}
On the other hand, we now show that the integral of $L_2^\epsilon$ is small uniformly in $\epsilon$. By the mean value theorem, and the fact that $\epsilon \leq r/ (2 \| \de f\|_{L^\infty})$,  we have
\begin{align}
\int_{\RR} |L_2^\epsilon(\eta,\zeta)| d\zeta
& \leq C \| \partial_{\eta\eta} f\|_{L^\infty} \int_{|\zeta - \eta| \leq r} \frac{d\zeta} { |\eta - \zeta - \epsilon \de f(\eta,t)|^{\beta}} \notag\\
&\leq C \| \partial_{\eta\eta} f\|_{L^\infty} \int_{|z| \leq 2r} \frac{dz} { |z|^{\beta}} \leq C \| \partial_{\eta\eta}f \|_{L^\infty}  r^{1-\beta}. \label{eq:L2eps:bound}
\end{align}
By combining \eqref{eq:L1eps:limit} and \eqref{eq:L2eps:bound} we conclude that for any $r>0$
\begin{align}
\left| \lim_{\epsilon\to 0} \int_{\RR} L^\epsilon(\eta,\zeta) d\zeta - \int_{|\eta-\zeta|>r}  \frac{(\eta-\zeta)   ( \de f(\zeta,t) - \de f (\eta,t) ) d\zeta}{((\eta-\zeta)^2+(f(\eta,t)-f(\zeta,t))^2)^{\frac{2+\beta}{2}}} \right| \leq C \| \partial_{\eta\eta}f \|_{L^\infty}  r^{1-\beta},
\end{align}
which converges to $0$ as $r \to 0$, since $\beta \in (0,1)$. We have thus proven that
\begin{align*}
& \lim_{\epsilon\to 0}\vv(\eta - \epsilon\de f(\eta,t),f(\eta,t) + \epsilon) \cdot (-\de f(\eta - \epsilon \de f(\eta,t),t),1) \notag\\
& \qquad \qquad =
\frac{\rho^2-\rho^1}{C_\beta} \int_{\RR}\frac{(\eta-\zeta)(\de f(\eta,t)-\de f(\zeta,t))
d\zeta}{((\eta - \zeta)^2 + (f(\eta,t) - f(\zeta,t))^2)^{\frac{2 + \beta}{2}}},
\end{align*}
and hence \eqref{ec:1} holds. The rest of the section is devoted to proving the following result.

\begin{theorem}\label{elpatch}
Let $\rho^2>\rho^1$, $\beta \in (0,1)$, and $f_0\in H^s$ for $s\geq 4$. Then there exists $T=T(\|f_0\|_{H^s})>0$ such that the contour differential equation given by \eqref{ec:1}--\eqref{ec:2} has a unique solution $f\in C([0,T],H^s)$.
\end{theorem}
We give the proof for $s=4$ and leave $s>4$ to the reader. For notational convenience, we take the coefficient $(\rho^2-\rho^1)/C_\beta=1$ and omit the time dependence of $f$.

\begin{proof}[Proof of Theorem~\ref{elpatch}] We proceed by proving an a priori energy estimate of the form
\begin{align*}
\frac{d}{dt}\|f\|_{H^4}\leq C (1 + \|f\|_{H^4})^k
\end{align*}
for $C$ and $k>1$ universal constants.
By \eqref{ec:1} we have
\begin{align*}
\frac12\frac{d}{dt}\|f\|_{L^2}^2=\int_{\RR}f(\eta)f_t(\eta)d\eta=I_1+I_2,
\end{align*}
where
\begin{align*}
I_1=\int_{|\zeta|>1} \zeta \int_{\RR}\frac{f(\eta)(\de f(\eta)-\de f(\eta-\zeta))
}{(|\zeta|^2 + (f(\eta) - f(\eta-\zeta))^2)^{\frac{2 + \beta}{2}}} d\eta d\zeta
\end{align*}
and
\begin{align*}
I_2=\int_{|\zeta|\leq 1} \zeta \int_{\RR}\frac{f(\eta)(\de f(\eta)-\de f(\eta-\zeta))
}{(|\zeta|^2  + (f(\eta) - f(\eta-\zeta))^2)^{\frac{2 + \beta}{2}}} d\eta d\zeta.
\end{align*}
Using the Cauchy-Schwartz inequality and $\beta>0$, we estimate
\begin{align*}
I_1\leq\int_{|\zeta|>1}\frac{1}{|\zeta|^{1+\beta}}\int_{\RR}|f(\eta)|(|\de f(\eta)|+|\de f(\eta-\zeta)|)
d\eta d\zeta\leq C \|f\|_{H^1}^2.
\end{align*}
On the other hand, using the Cauchy-Schwartz inequality and  the Gagliardo characterization of the Sobolev norm~\cite{Stein} we obtain
\begin{align*}
I_2 &\leq \int_{|\zeta|\leq 1}  \int_{\RR}\frac{|f(\eta)|\; |\de f(\eta)-\de f(\eta-\zeta)|
}{|\zeta|^{1+\beta}} d\eta d\zeta\\
&\leq \|f \|_{L^2} \int_{|\zeta|\leq 1} \frac{1}{|\zeta|^{1+\beta}} \left(\int_{\RR} |\de f(\eta)-\de f(\eta-\zeta)|^2 d\eta \right)^{1/2} d\zeta\\
&\leq C \|f \|_{L^2} \left( \int_{|\zeta|\leq 1}  \int_{\RR} \frac{|\de f(\eta)-\de f(\eta-\zeta)|^2}{|\zeta|^{2+2\beta}}  d\eta d\zeta  \right)^{1/2}   \leq C \|f\|_{L^2} \|f\|_{H^{1+\beta}}.
\end{align*}
Consequently
\begin{equation}\label{ee1}
\frac12\frac{d}{dt}\|f\|_{L^2}^2\leq C\|f\|_{H^2}^2.
\end{equation}

Next we estimate the $\dot{H}^4$ norm of $f$. Using the Leibniz rule we obtain
\begin{align*}
\frac12\frac{d}{dt}\|\de^4 f\|_{L^2}^2(t)=\int_{\RR}\de^4f(\eta)\de^4f_t(\eta)d\eta=J_0+J_1+J_2+J_3,
\end{align*}
where
\begin{equation}\label{ji}
J_i=\binom{3}{i}\int_{\RR}\de^4f(\eta)\de \int_{\RR}\zeta(\de^{4-i}f(\eta) - \de^{4-i} f(\eta - \zeta))
\de^{i}\Big(\zeta^2+(f(\eta)-f(\eta - \zeta))^2)^{-\frac{2+\beta}{2}}\Big)d\zeta d\eta
\end{equation}
for $0 \leq i \leq 3$. We deal first with the most singular term $J_0$, which contains fifth order derivatives. Integrating by parts we rewrite
\begin{align} \label{ji-1-1}
  J_0
& = \int_{\RR}
      \de^4 f(\eta)
      \de \int_{\RR}
            \frac{\zeta (\de^4 f(\eta)- \de^4 f(\eta - \zeta))}
                 {(\zeta^2+(f(\eta)-f(\eta - \zeta))^2)^{\frac{2+\beta}{2}}} d\zeta
                 d\eta \notag
\\
& = - \int_{\RR}
      \de^5 f(\eta)
      \int_{\RR}
        \frac{\zeta (\de^4 f(\eta)- \de^4 f(\eta - \zeta))}
             {(\zeta^2+(f(\eta)-f(\eta - \zeta))^2)^{\frac{2+\beta}{2}}}
      d\zeta d\eta.
\end{align}
By the change of variable $(\eta, \zeta) \to (\eta-\zeta,
-\zeta)$, we can further rewrite $J_0$ as
\begin{align}\label{ji-1-2}
  J_0
& = - \int_{\RR}
      \int_{\RR}
      \de^5 f(\eta-\zeta)
        \frac{\zeta (\de^4 f(\eta)- \de^4 f(\eta - \zeta))}
             {(\zeta^2+(f(\eta)-f(\eta - \zeta))^2)^{\frac{2+\beta}{2}}}
      d\zeta d\eta.
\end{align}
Taking the average of  \eqref{ji-1-1} and \eqref{ji-1-2}, and
applying a further change of variables $(\eta, \zeta) \to
(\eta+\frac{\zeta}{2}, \zeta)$, we have
\begin{align*}
  J_0
& = - \frac 12
     \int_{\RR} \int_{\RR}
      (\de^5 f(\eta) + \de^5 f(\eta-\zeta))
        \frac{(\de^4 f(\eta)- \de^4 f(\eta - \zeta)) \zeta}
             {(\zeta^2+(f(\eta)-f(\eta - \zeta))^2)^{\frac{2+\beta}{2}}}
      d\zeta d\eta
\\
& = - \frac 12
     \int_{\RR} \int_{\RR}
      (\de^5 f(\eta + \tfrac{\zeta}{2}) + \de^5 f(\eta-\tfrac{\zeta}{2}))
        \frac{(\de^4 f(\eta + \frac{\zeta}{2})- \de^4 f(\eta - \frac{\zeta}{2})) \zeta}
             {(\zeta^2+(f(\eta + \frac{\zeta}{2})-f(\eta - \frac{\zeta}{2}))^2)^{\frac{2+\beta}{2}}}
      d\zeta d\eta
\\
& = - \int_{\RR} \int_{\RR}
      \partial_\zeta (\de^4 f(\eta + \tfrac{\zeta}{2}) - \de^4 f(\eta-\tfrac{\zeta}{2}))
        \frac{(\de^4 f(\eta + \frac{\zeta}{2})- \de^4 f(\eta - \frac{\zeta}{2})) \zeta}
             {(\zeta^2+(f(\eta + \frac{\zeta}{2})-f(\eta - \frac{\zeta}{2}))^2)^{\frac{2+\beta}{2}}}
      d\zeta d\eta.
\end{align*}
Here we used that $\de^5 f(\eta +\zeta/2) = 2 \dz \de^4 f(\eta+\zeta/2)$ and  $\de^5 f(\eta -\zeta/2) = -2 \dz \de^4 f(\eta-\zeta/2)$. Integration by parts in $\zeta$ and the change of variables $(\eta,\zeta) \to (\eta-\zeta/2,\zeta) $ then gives
\begin{align}
  J_0
& = \frac 12
    \int_{\RR} \int_{\RR}
      (\de^4 f(\eta + \tfrac{\zeta}{2}) - \de^4 f(\eta-\tfrac{\zeta}{2}))^2
       \dz \left(\zeta
                 (\zeta^2+(f(\eta + \tfrac{\zeta}{2})-f(\eta -
                 \tfrac{\zeta}{2}))^2)^{-\frac{2+\beta}{2}} \right)
      d\zeta d\eta \notag
\\
& = \frac 12
    \int_{\RR} \int_{\RR}
      (\de^4 f(\eta) - \de^4 f(\eta-\zeta))^2
       \dz \left(\zeta
                 (\zeta^2+(f(\eta)-f(\eta-\zeta))^2)^{-\frac{2+\beta}{2}} \right)
      d\zeta d\eta. \label{ji-2}
\end{align}
We compute explicitly the derivative term and obtain
\begin{align*}
&
  \dz\left(\zeta (\zeta^2 + (f(\eta) - f(\eta -\zeta))^2)^{\frac{2+\beta}{2}}\right)
\\
& = \frac{(2+\beta)(f(\eta) - f(\eta - \zeta))(f(\eta) - f(\eta -
          \zeta) - \de f(\eta - \zeta)\zeta)}
         {(\zeta^2 + (f(\eta) - f(\eta - \zeta))^2)^{\frac{4+\beta}{2}}}
   - \frac{1+\beta}{(\zeta^2 + (f(\eta) - f(\eta - \zeta))^2)^{\frac{2+\beta}{2}}}.
\end{align*}
Applying the explicit form of the derivative in \eqref{ji-2} gives
\begin{align} \label{J-0}
  J_0
& =  J_{0}^{(1)} + J_0^{(2)},
\end{align}
where
\begin{align*}
&J_{0}^{(1)} = \frac{2+\beta}{2}
    \int_{\RR}\int_{\RR}
      (\de^4 f(\eta) - \de^4 f(\eta-\zeta))^2
      \frac{(f(\eta) - f(\eta-\zeta))(f(\eta) - f(\eta -
          \zeta) - \de f(\eta-\zeta)\zeta)}
         {(\zeta^2 + (f(\eta)-f(\eta-\zeta))^2)^{\frac{4+\beta}{2}}}
    d\zeta d\eta
\end{align*}
and
\begin{align*}
J_0^{(2)} = -\frac{1+\beta}{2}
    \int_{\RR} \int_{\RR}
      \frac{(\de^4 f(\eta) - \de^4 f(\eta-\zeta))^2}
           {(\zeta^2 + (f(\eta) - f(\eta - \zeta))^2)^{\frac{2+\beta}{2}}}
      d\zeta d\eta.
\end{align*}
Using Gagliardo's characterization of the Sobolev norm, we may bound $J_0^{(1)}$ as
\begin{align}
J_{0}^{(1)}
&\leq  \frac{2+\beta}{2}
    \int_{\RR}\int_{\RR}
      (\de^4 f(\eta) - \de^4 f(\eta-\zeta))^2
      \frac{|f(\eta) - f(\eta-\zeta)| \; |f(\eta) - f(\eta -
          \zeta) - \de f(\eta-\zeta)\zeta|}
         {(\zeta^2 + (f(\eta)-f(\eta-\zeta))^2)^{\frac{4+\beta}{2}}}
    d\zeta d\eta \notag\\
 &\leq C \| f\|_{W^{1,\infty}} \|f\|_{W^{2,\infty}}  \int_{\RR}\int_{\RR}
      \frac{(\de^4 f(\eta) - \de^4 f(\eta-\zeta))^2}
         {|\zeta|^{1+\beta}}
    d\zeta d\eta \notag\\
 &= C \| f\|_{W^{1,\infty}} \|f\|_{W^{2,\infty}}  \| \Lambda^{\frac{\beta}{2}} \de^4 f\|_{L^2}^2. \label{eq:J01:A}
\end{align}
Applying the Sobolev embedding $H^4 \subset W^{2,\infty}$, and interpolation, we obtain from \eqref{eq:J01:A} that
\begin{align} \label{J-0-1}
  J_{0}^{(1)}
\leq
  C \|f\|_{H^4}^{4 + 2\beta} (1 + \|f\|_{H^4}^{2+\beta})
  + \frac{\|\Lambda^{\frac{1+\beta}{2}}\de^4f\|^2_{L^2}}
         {8(1+\|\de f\|^2_{L^\infty})^{\frac{2+\beta}{2}}}.
\end{align}

Next we show that the first term of $J_0^{(2)}$ in \eqref{J-0} gives a
dissipation, which is in fact needed to close the energy estimate. To this
end, we separate this term as follows:
\begin{align} \label{J-0-dissipation}
J_0^{(2)} &=
  -\frac{1+\beta}{2}
    \int_{\RR} \int_{\RR}
      \frac{(\de^4 f(\eta) - \de^4 f(\eta-\zeta))^2}
           {(\zeta^2 + (f(\eta) - f(\eta - \zeta))^2)^{\frac{2+\beta}{2}}}
      d\zeta d\eta \notag
\\
&=  -\frac{1+\beta}{2}
    \int_{\RR} \int_{\RR}
      \frac{(\de^4 f(\eta) - \de^4 f(\eta-\zeta))^2}
           {|\zeta|^{2+\beta}} \; \frac{d\zeta d\eta }{ (1+ ( (f(\eta) - f(\eta - \zeta))/\zeta)^2)^{\frac{2+\beta}{2}}}
      \notag
\\
& = -\frac{1+\beta}{2(1+\|\de f\|_{L^\infty}^2)^{\frac{2+\beta}{2}}}
    \int_{\RR} \int_{\RR}
      \frac{(\de^4 f(\eta) - \de^4 f(\eta-\zeta))^2}
           {|\zeta|^{2+\beta}}
      d\zeta d\eta
    + J_{01}^{(2)}
    + J_{02}^{(2)}\notag
\\
& = -\frac{1+\beta}{2(1+\|\de f\|_{L^\infty}^2)^{\frac{2+\beta}{2}}}
    \|\Lambda^{\frac{1+\beta}{2}}\de^4 f\|_{L^2}^2
    + J_{01}^{(2)}
    + J_{02}^{(2)},
\end{align}
where
\begin{align*}
  J_{01}^{(2)}
  = -\frac{1+\beta}{2}
    \int_{\RR}\int_{\RR}
     & \frac{(\de^4 f(\eta) - \de^4 f(\eta-\zeta))^2}{|\zeta|^{2+\beta}} \\
      &\times \left((1 + ((f(\eta)-f(\eta-\zeta))/\zeta)^2)^{-(2+\beta)/2} - (1+|\de f|^2)^{-(2+\beta)/2}\right)
    d\zeta d\eta,
\end{align*}
and
\begin{align*}
  J_{02}^{(2)}
  = -\frac{1+\beta}{2}
    \int_{\RR}\int_{\RR}
      \frac{(\de^4 f(\eta) - \de^4 f(\eta-\zeta))^2}{|\zeta|^{2+\beta}}
      \left((1+|\de f|^2)^{-\frac{2+\beta}{2}}
       - (1+\|\de f\|_{L^\infty}^2)^{-\frac{2+\beta}{2}}\right)
    d\zeta d\eta.
\end{align*}
Note that $J_{02}^{(2)}$ is well-defined for $f \in
H^{4+\frac{1+\beta}{2}}$, and we clearly have
\begin{align} \label{J-0-3}
  J_{02}^{(2)} \leq 0 \,.
\end{align}
Thus we only need a bound for $J_{01}^{(2)}$. By a similar argument as in
the estimates for $J_{0}^{(1)}$,
we estimate
\begin{align*}
\left|(1 + ((f(\eta)-f(\eta-\zeta))/\zeta)^2)^{-(2+\beta)/2} - (1+|\de f|^2)^{-(2+\beta)/2}\right| \leq C \|f\|_{W^{2,\infty}} |\zeta|
\end{align*}
and therefore
\begin{align}
  J_{01}^{(2)}
& \leq C \|f \|_{W^{2,\infty}} \| \Lambda^{\frac{\beta}{2}} \de^4 f \|_{L^2}^2
\leq  C \|f\|_{H^4}^{3 + 2\beta} (1 + \|f\|_{H^4}^{2+\beta})+ \frac{(1+\beta)\|\Lambda^{\frac{1+\beta}{2}}\de^4f\|^2_{L^2}}{8(1+\|\de f\|^2_{L^\infty})^{\frac{2+\beta}{2}}} \,.\label{J-0-2}
\end{align}
Combining \eqref{J-0-dissipation}, \eqref{J-0-3}, and \eqref{J-0-2},
we obtain
\begin{align} \label{J-0-first}
J_0^{(2)}
& \leq
   -\frac{1+\beta}{4(1+\|\de f\|_{L^\infty}^2)^{\frac{2+\beta}{2}}}
    \|\Lambda^{\frac{1+\beta}{2}}\de^4 f\|_{L^2}^2
    + C\|f\|_{H^4}^{3 + 2\beta} (1 + \|f\|_{H^4}^{2+\beta}).
\end{align}
Inserting the bounds \eqref{J-0-1} and \eqref{J-0-first} into \eqref{J-0}, we
finally obtain that
\begin{align} \label{J-0-second}
  J_0
\leq
   -\frac{1+\beta}{8(1+\|\de f\|_{L^\infty}^2)^{\frac{2+\beta}{2}}}
    \|\Lambda^{\frac{1+\beta}{2}}\de^4 f\|_{L^2}^2
    + \|f\|_{H^4}^{3 + 2\beta} (1 + \|f\|_{H^4}^{3+\beta}).
\end{align}

Estimating $J_i$, with $1\leq i\leq 3$, is direct since in these terms the higher order partial derivative in $f$ one can find is $4$. We may bound
\begin{align*}
\sum_{i=1}^3J_i\leq C \|f\|_{H^4}^{3 + 2\beta} (1 + \|f\|_{H^4}^{3+\beta})  +\frac{(1 + \beta)\|\Lambda^{\frac{1+\beta}2}\de^4f\|^2_{L^2}}{8(1 + \|\de f\|^2_{L^\infty})^{\frac{2+\beta}{2}}}
\end{align*}
which together with the estimate for  $J_0$ arising from \eqref{J-0-second} gives
\begin{align*}
\frac12\frac{d}{dt}\|\de^4 f\|_{L^2}^2\leq C \|f\|_{H^4}^{3 + 2\beta} (1 + \|f\|_{H^4}^{3+\beta})  -\frac{(1 + \beta)\|\Lambda^{\frac{1+\beta}2}\de^4f\|^2_{L^2}}{8(1 + \|\de f\|^2_{L^\infty})^{\frac{2+\beta}{2}}}.
\end{align*}
Finally, using the $L^2$ estimate \eqref{ee1} we obtain
\begin{align}
\frac12\frac{d}{dt}\|f\|_{H^4}^2\leq C \|f\|_{H^4}^{3 + 2\beta} (1 + \|f\|_{H^4}^{3+\beta})   -\frac{(1 + \beta)\|\Lambda^{\frac{1+\beta}2}\de^4f\|^2_{L^2}}{8(1 + \|\de f\|^2_{L^\infty})^{\frac{2+\beta}{2}}}. \label{eq:H4:BOUND}
\end{align}
This concludes the a priori estimates needed to obtain the local-existence of smooth solutions to \eqref{ec:1}--\eqref{ec:2}.

Regarding uniqueness, consider two solutions $f_1$ and $f_2$ of \eqref{ec:1}--\eqref{ec:2}. Using the nonlinear dissipation, a similar approach allows us to get
\begin{align*}
\frac12\frac{d}{dt}\|f_1-f_2\|_{L^2}^2\leq C(\|f_1\|_{H^4},\|f_2\|_{H^4})\|f_1-f_2\|_{L^2}^2 -\frac{(1 + \beta)\|\Lambda^{\frac{1+\beta}2}(f_1-f_2)\|^2_{L^2}}{8(1 + \|\de f_1\|^2_{L^\infty})^{\frac{2+\beta}{2}}}
\end{align*}
for some polynomial $C(\cdot,\cdot)$. The Gr\"onwall inequality then concludes the proof of uniqueness of solutions to \eqref{ec:1}--\eqref{ec:2}.

The construction of solutions obeying the a priori estimate \eqref{eq:H4:BOUND}, and verifying that the solutions to \eqref{ec:1}--\eqref{ec:2} give a weak solution of patch-type to the SIPM equations in the sense of Definition~\ref{dsd}, follows precisely the same arguments as in the case of the classical porous media equation (i.e. $\beta = 0$). We omit these details here and refer the reader to \cite{DP,DPR}.
\end{proof}

\appendix

\section{Proof of abstract results about continued fractions}
\label{app:fractions}
The goal of this appendix is to give a proof of Theorem~\ref{thm:*}. Recall that $\{ p_n \}_{n\geq 1}$ is a {\em strictly increasing} sequence of {\em positive} numbers, and we need to {find} and {estimate} solutions $\lambda$ of the continued fraction equation
\begin{align}
\lambda p_1 = \frac{1}{ \lambda p_2 - \frac{1}{ \lambda p_3 - \frac{1}{ \lambda p_4 - \dots }}}
\label{eq:*}
\end{align}
which are {\em real} and {\em positive}. In addition, given such a solution $\lambda$ of \eqref{eq:*} we recursively solve
\begin{align}
 \eta_{n+1} = - \lambda p_n - \frac{1}{\eta_n} \mbox{ for all } n\geq 2 \label{eq:eta:rec}
\end{align}
with initial condition $\eta_2 = - \lambda p_1$, and then denote
\begin{align}
 c_n = p_n \eta_n \ldots \eta_2 \label{eq:cn:def}
\end{align}
for all $n \geq 2$, and $c_1 = p_1$.

\begin{proof}[Proof of Theorem~\ref{thm:*}]
For $n\geq 2$, and $\lambda \geq 2/p_n$ one may define the function
\begin{align}
G_n (\lambda) = \frac{\lambda p_n - \sqrt{ \lambda^2 p_n^2 - 4}}{2} = \frac{2}{\lambda p_n + \sqrt{ \lambda^2 p_n^2 - 4}}. \label{eq:Gn:def}
\end{align}
Formally, we may write
\begin{align*}
 G_n(\lambda) = \frac{1}{ \lambda p_n - \frac{1}{ \lambda p_n - \frac{1}{ \lambda p_n - \dots }}}
\end{align*}
for any $\lambda \in \DD(G_n) = [2/p_n,\infty)$. Henceforth $\DD$ will stand for {\em domain}. Note that $\DD(G_{n}) \subset \DD(G_{n+1})$ for any $n \geq 2$. Let observe some properties of the $G_n$'s,  that are all due in view of \eqref{p:cond:i}: $ G_n (\lambda) \geq 0$; $G_n (\lambda) \to 0 \mbox{ as } \lambda \to \infty$; $1/(\lambda p_n) < G_n(\lambda) < 2/(\lambda p_n)$; and $G_{n+1} (\lambda) < G_{n}(\lambda)$ for any $\lambda \in \DD(G_n)$.
Define
\begin{align*}
 \lambda_0 = \frac{1}{\sqrt{p_1 p_2}} \mbox{, } \lambda_1 = \frac{2}{p_2} \mbox{, and } \DD_0= [ \min\{ \lambda_0,\lambda_1\}, \infty).
\end{align*}
One property of the $G_n$'s which we will use later is that
\begin{align}
G_{n+1}(\lambda) < \lambda p_n \mbox{ for any } \lambda \in
\DD(G_{n+1}) \cap \DD_0
 \label{eq:Gn:hard}
\end{align}
for any $n\geq 2$. Indeed, for those $n$ such that $ p_{n+1} < 2
p_n$, we have
\begin{align*}
 G_{n+1}(\lambda) < \frac{2}{\lambda p_{n+1}} < \lambda p_n,
\qquad
 \text{for any $\lambda \in \DD(G_{n+1})$ },
\end{align*}
since $p_{n+1} < 2 p_n \Rightarrow 2 < \lambda^2 p_n p_{n+1}$ when
$\lambda \geq 2/p_{n+1}$. If instead $p_{n+1} \geq 2 p_n$, then
for any $\lambda \in \DD_0$,
\begin{align*}
  \lambda^2 p_n p_{n+1}
\geq
  \min \left\{\frac{1}{p_1 p_2} p_n p_{n+1}, \; \frac{4}{p_2^2} p_n p_{n+1} \right\}
\geq
  \min \left\{\frac{2 p_n^2}{p_1 p_2}, \; 4 \right\}
\geq 2 \,,
\end{align*}
which proves~\eqref{eq:Gn:hard}.

For $n\geq 2$ we now introduce an auxiliary function $F_n(\lambda)$
formally defined as
\begin{align}
 F_n (\lambda) = \frac{1}{ \lambda p_n - \frac{1}{ \lambda p_{n+1} - \frac{1}{ \lambda p_{n+2} - \dots }}} \label{eq:Fn:def}
\end{align}
for any $\lambda \in \DD_0$.
In order to define $F_n$ rigorously, for $n\geq 2$ and $k \geq 0$, define the (real) rational function
\begin{align*}
F_{n,k}(\lambda) = \frac{1}{ \lambda p_n - \frac{1}{ \lambda p_{n+1} - \cdots \frac{1}{ \lambda p_{n+k} - G_{n+k+1}(\lambda)}}}
\end{align*}
for any $\lambda \geq 2/p_{n+k+1}$. After a certain value of $k$ we have that all the rational functions $F_{n,k}$ are defined on $\DD_0$ modulo a finite set of points where there are vertical asymptotes. Then we define
\begin{align*}
 F_n(\lambda) = \lim_{k \to \infty} F_{n,k}(\lambda)
\end{align*}
for all $\lambda$ in $\DD_0$ where this limit exits. At points where this limit does not exits, $F_n$ has vertical asymptotes.

Note that the equation we wish to solve, \eqref{eq:*} may now be written in this language as
\begin{align}
\lambda p_1 = F_2 (\lambda) \label{eq:**}.
\end{align}
In addition,  if $\lambda_*$ is a solution of \eqref{eq:**} above, we have
\begin{align}
F_{n+1}(\lambda_*) = \lambda_* p_n - \frac{1}{F_n(\lambda_*)} \label{eq:Fn:rec}
\end{align}
for $n\geq 2$. Therefore $F_n(\lambda_*) = - \eta_n$, where $\eta_n$ was defined in \eqref{eq:eta:rec} above.

We now show that $F_n(\lambda)$ is well-defined, continuous, and non-increasing for $ \lambda \in \DD(G_{n}) \cap \DD_0$. For such a fixed $\lambda$ and $n \geq 2$, we have by \eqref{eq:Gn:hard} that $\lambda p_{n+k} - G_{n+k+1}(\lambda)>0$ for any $k \geq 0$. By \eqref{eq:Gn:def} we also have $\lambda p_{n+k} - G_{n+k}(\lambda) > 0$ and recalling $G_{n+k+1}(\lambda) < G_{n+k}(\lambda)$ leads to
\begin{align}
 0<  \frac{1}{\lambda p_{n+k} - G_{n+k+1}(\lambda)}
  < \frac{1}{\lambda p_{n+k} - G_{n+k}(\lambda)}
  = G_{n+k}(\lambda). \label{eq:TO:iterate}
\end{align}
By iterating \eqref{eq:TO:iterate}, in the and recalling the definition of $F_{n,k}$ one may show that\begin{align*}
0 < F_{n,k+1}(\lambda) < F_{n,k}(\lambda) < G_n(\lambda)
\end{align*}
holds for any $k\geq 0$, and any $\lambda \in \DD(G_n)\cap \DD_0$. Due to the positivity of the $F_{n,k}$'s and using that
\begin{align*}
 F_{n,k}(\lambda)  = \frac{1}{\lambda p_n - F_{n+1,k-1}(\lambda)}
\end{align*}
we in fact additionally obtain that
$
 1/(\lambda p_n) < F_{n,k}(\lambda)
$
for all $k \geq 0$. At last, we notice that since $G_{n+k+1}(\lambda)$ is decreasing in $\lambda$, one may show that $F_{n,k}(\lambda)$ is decreasing as well. Collecting the above obtained information of $F_{n,k}$, we may hence conclude that the function $F_n(\lambda)$ is well-defined, continuous, non-increasing and satisfies the bound
\begin{align}
 \frac{1}{\lambda p_n} < F_{n}(\lambda) < G_n(\lambda) < \frac{2}{\lambda p_n} \label{eq:Fn:Gn:bound}
\end{align}
on the set $\DD(G_n) \cap \DD_0$.

In order to solve $\lambda p_1 = F_2(\lambda)$ we distinguish two cases, depending on the relative size of $\lambda_0$ and $\lambda_1$.  The direct case is $\lambda_1 \leq \lambda_0$, equivalently $p_2 \geq 4 p_1$. In this case, inserting $n=2$ in \eqref{eq:Fn:Gn:bound}
yields
\begin{align*}
 \frac{1}{\lambda p_2} < F_2 (\lambda) < G_2(\lambda)
\end{align*}
for $\lambda \geq 2/p_2 = \lambda_1$; thus $F_2$ is continuous on $\DD_0$. Due to the above inequality it is natural to check where $\lambda p_1$ intersects $1/(\lambda p_2)$ and $G_2(\lambda)$. The graph of $\lambda p_1$ intersects the graph of $1/(\lambda p_2)$ at $\lambda_0 = 1/\sqrt{p_1 p_2}$ and the graph of $G_2(\lambda)$ at $1/\sqrt{p_1 p_2 - p_1^2}$.  By the intermediate value theorem, there exists $
\frac{1}{\sqrt{p_1 p_2}} < \lambda_* < 1 / \sqrt{p_1 p_2 - p_1^2}
$
such that $\lambda_* p_1 = F_2 (\lambda_*)$. When $\lambda_1 \leq \lambda_0$ we have thus found a solution $\lambda_*$ to \eqref{eq:*}.

\begin{figure}[!h]
  \includegraphics[width=68ex]{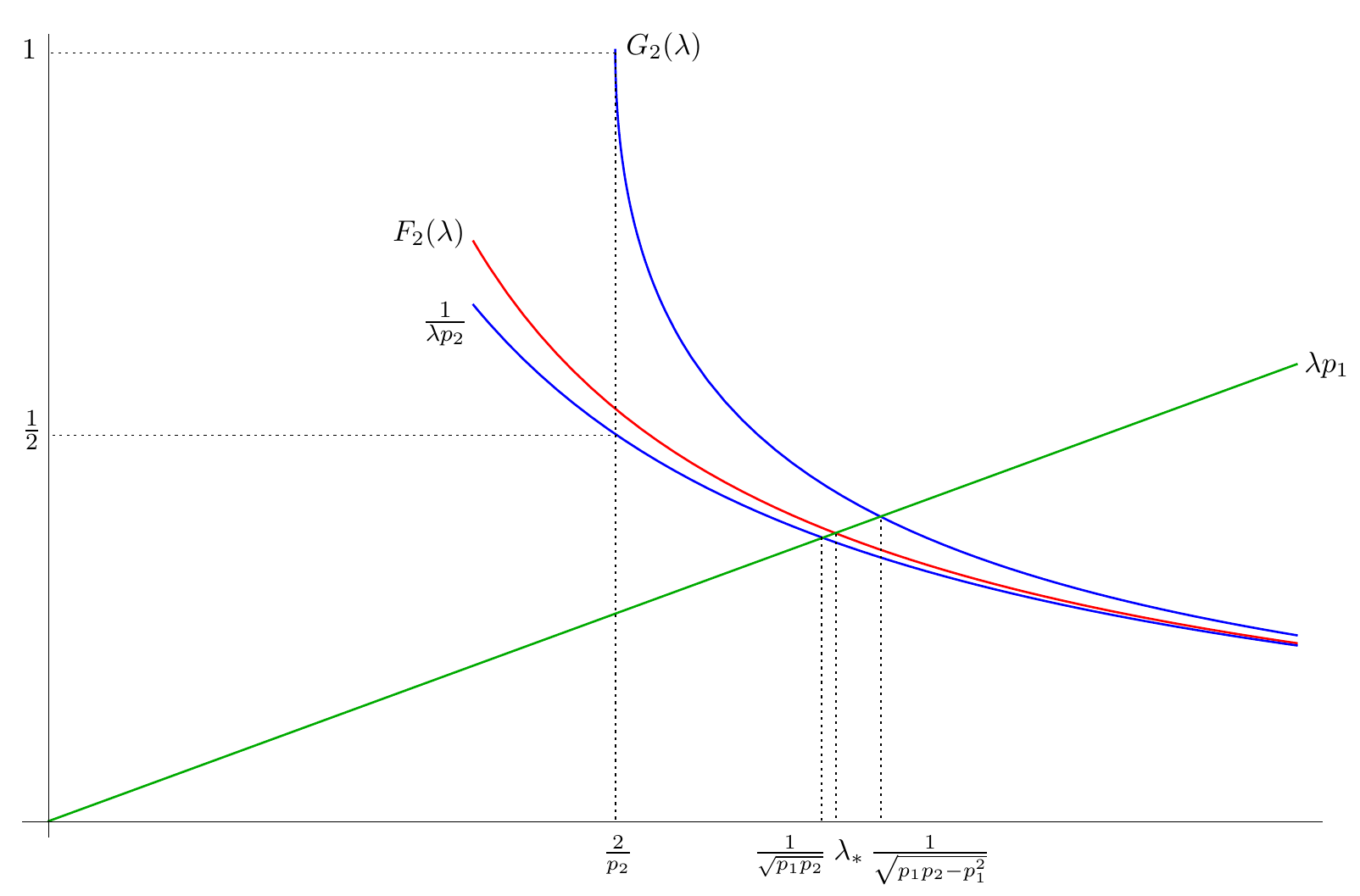}
  \caption{Solving for $\lambda_*$ when $2/p_{2} \leq 1/\sqrt{p_{1} p_{2}}$. The above plot was obtained by numerically computing $F_2$, where the coefficients $p_n \approx n^4 $ arise in the study of the MG equation.  In the numerical computation we have used the explicit formula for $p_n$ cf.~\cite{FV2} equation (2.33), and set all physical parameters to $1$.}
  \label{fig:MG}
\end{figure}

For the $\lambda_*$ found above, it is left to estimate the coefficients $|c_{n}|$ as defined in \eqref{eq:cn:def}, i.e.
\begin{align*}
 |c_n|  = p_n F_n(\lambda_*) \ldots F_2(\lambda_*).
\end{align*}
But since $\lambda_* \geq \lambda_1 = 2/p_2$, we may use the bound \eqref{eq:Fn:Gn:bound}, and therefore
\begin{align*}
 |c_n| \leq \frac{2^{n-1}}{\lambda_*^{n-1} p_{n-1} \ldots p_2 } \leq \frac{p_2^{n-1}}{p_{n-1} \ldots p_2}.
\end{align*}
Since the $p_n$s are grow unboundedly, there exits an $n_0 \geq 2$
such that $p_{n_0} \leq 2 p_2 < p_{n_0+1}$. Due to the monotonicity
of the $p_n$s, we may therefore bound
\begin{align}
|c_n| \leq p_2 \label{eq:cn:MG:low}
\end{align}
for all $n\leq n_0$ and
\begin{align}
 |c_n| \leq \frac{p_2^{n-1}}{p_{n-1} \ldots p_{n_0+1} p_2^{n_0-1}} \leq \frac{p_2^{n-n_0}}{p_{n_0+1}^{n-n_0-1}} \leq p_2 2^{n_0+1 -n}  \label{eq:cn:MG:high}
\end{align}
for all $n\geq n_0+1$, which shows that the $c_n$s eventually decay exponentially. Moreover,  from \eqref{eq:cn:MG:low} and \eqref{eq:cn:MG:high} we obtain
\begin{align*}
\| n^s c_n\|_{\ell^2(\NN)}^2 \leq n_0^{2s} \sum_{n=0}^{n_0} c_n^2 + p_2^2 \sum_{k\geq 0} 2^{-2k} \leq n_0^{2s} \| c_n\|_{\ell^2(\NN)}^2 + p_2^2  \leq C  \left( n_0^{2s} + \frac{p_2^2}{p_1^2} \right) \| c_n\|_{\ell^2(\NN)}^2,
\end{align*}
which proves \eqref{eq:Sobolev:Estimate}. Here we used that $c_1 = p_1 = p_2 (p_1/p_2)$.

The case $\lambda_0 < \lambda_1$, or equivalently $p_2 < 4 p_1$ is more involved, since it will not be possible to show in general that $F_2$ is continuous on $\DD_0 = [\lambda_0,\infty)$; we only know this on $[\lambda_1, \infty)$. To overcome this we shall estimate the largest value of $\lambda$ where $F_2$ has an asymptote, and work to the right of it. Due to \eqref{p:cond:ii} there exits $n_0 \geq 2$ such that
\begin{align*}
 \frac{2}{p_{n_0+1}} \leq \lambda_0 < \frac{2}{p_{n_0}} \leq \lambda_1.
\end{align*}
Hence, by \eqref{eq:Fn:Gn:bound} we have that $F_{n_0 + 1}(\lambda)$ is well-defined, continuous, and non-increasing on $\DD_0 = [\lambda_0, \infty)$. Moreover, by combining \eqref{eq:Fn:Gn:bound} with \eqref{eq:Gn:hard} we see that
\begin{align*}
 0 < F_{n_0+1}(\lambda) < \lambda p_{n_0}
\end{align*}
on $\DD_0$. This hence allows to define the function $F_{n_0}(\lambda)$ on all of $\DD_0$, by setting
\begin{align}
 F_{n_0}(\lambda) = \frac{1}{\lambda p_{n_0} - F_{n_0+1}(\lambda)}. \label{eq:Fn0}
\end{align}
Note that \eqref{eq:Fn:Gn:bound} only gives us that $F_{n_0}$ is continuous, positive, and non-increasing on $[2/p_{n_0},\infty)$, but in view of \eqref{eq:Fn0} we now know these properties for $F_{n_0}$ on all of $\DD_0$. Our goal is to iterate this process and define $F_{n_0-1}, \ldots, F_2$ on some large enough set. To achieve this we need to stay to the ``right'' of vertical asymptotes. We inductively define the sets
\begin{align}
A_j = \{ \lambda \in \DD_0 \colon 0 < F_{j+1}(\lambda) < \lambda p_j\} \supset [ \frac{2}{p_j},\infty)\label{eq:Aj}
\end{align}
and on $A_j$ we define the continuous, non-decreasing function
\begin{align}
 F_j(\lambda) =  \frac{1}{\lambda p_{j} - F_{j+1}(\lambda)} \label{eq:Fnj}
\end{align}
for all $j \in \{2, n_0\}$. The induction starts at $j = n_0$, a case which was described in \eqref{eq:Fn0} and the paragraph below it. Intuitively speaking, if $a_j = \inf A_j > \lambda_0$, since $F_{j+1}$ is non-increasing, as $\lambda \to a_j +$, we have that $F_j(\lambda) \to + \infty$, i.e. $a_j$ is the largest vertical asymptote of $F_j$. We observe that $A_j$ is connected (that is, an interval) since $F_{j+1}$ is monotone and continuous on $A_j$. We either have that $A_j = \DD_0$ when $a_j = \lambda_0$, or $A_j = (a_j,\infty)$, with $a_j \in (\lambda_0, 2/p_j)$. Note that the monotonicity of $F_{j}$ follows from the monotonicity of $F_{j+1}$, which holds by induction, and from \eqref{eq:Fnj}.  We also note that by construction we have $A_2 \subset \ldots \subset A_{n_0} = \DD_0$; indeed  if $\lambda \in A_{j-1}$, then $0< F_j(\lambda)< \lambda p_{j-1}$, and hence by \eqref{eq:Fnj} we must have $0 < F_{j+1}(\lambda) < \lambda p_j$, and so $\lambda \in A_j$. Hence, we finally obtain that
\begin{align*}
 A_2 = \bigcap\limits_{j=2}^{n_0} A_j
\end{align*}
is an interval which either equals $\DD_0$, or it equals $(a_2,\infty)$ for some $a_2 \in (\lambda_0, 2/p_2)$. Since $F_3(\lambda)>0$ on $A_3 \supset A_2$, by \eqref{eq:Fnj} we obtain that
\begin{align}
\frac{1}{\lambda p_2} < F_2(\lambda) \label{eq:F2:bound:lower}
\end{align}
on $A_2$. At last due to the monotonicity and continuity of $F_2$ on
$A_2$, we obtain that the graph of $\lambda p_1$ must intersect the
graph of $F_2(\lambda)$ at some $\lambda_* \in A_2$. Indeed, if $A_2
=(a_2,\infty)$, $a_2 > \lambda_0$, then $\lim_{\lambda \to a_2+}
F_2(\lambda) = + \infty$ and $\lim_{\lambda \to \infty} F_2(\lambda)
= 0$, so we obtain the desired intersection point $\lambda_* > a_2$
from the intermediate value theorem. Otherwise, if $A_2 = \DD_0 =
[\lambda_0,\infty)$, we note that $\lambda p_1 = 1/(\lambda p_2)$ at
$\lambda = \lambda_1$, and so $\lambda_*$ such that $\lambda_* p_1 =
F_2(\lambda_*)$ must exist by \eqref{eq:F2:bound:lower} and the
intermediate value theorem. In either case we have obtained a
$\lambda_* > \lambda_0$ that solves \eqref{eq:*}.  In terms of upper
bounds, either we have $\lambda_* \leq \lambda_1 = 2/p_2$, or else
we use that $ \lambda_* \in [2/p_2,\infty)= \DD(G_2)$, and here we
have $F_2(\lambda) < G_2(\lambda)$ and hence $\lambda_* \leq
1/\sqrt{p_1 p_2 - p_1^2}$, the intersection point of $\lambda p_1$
with $G_2(\lambda)$. Note that the later case can only occur if $p_2
\geq 2 p_1$, and in this case we can further estimate $\lambda_*
\leq 1/p_1$. Thus, in general, we have obtained the upper bound
$\lambda_* \leq 1/\sqrt{p_1 p_2 - p_1^2}$ stated in the theorem, and
also the bound $\lambda_* \leq \max\{ 2/p_2, 1/p_1\}$.

\begin{figure}[!h]
  \includegraphics[width=68ex]{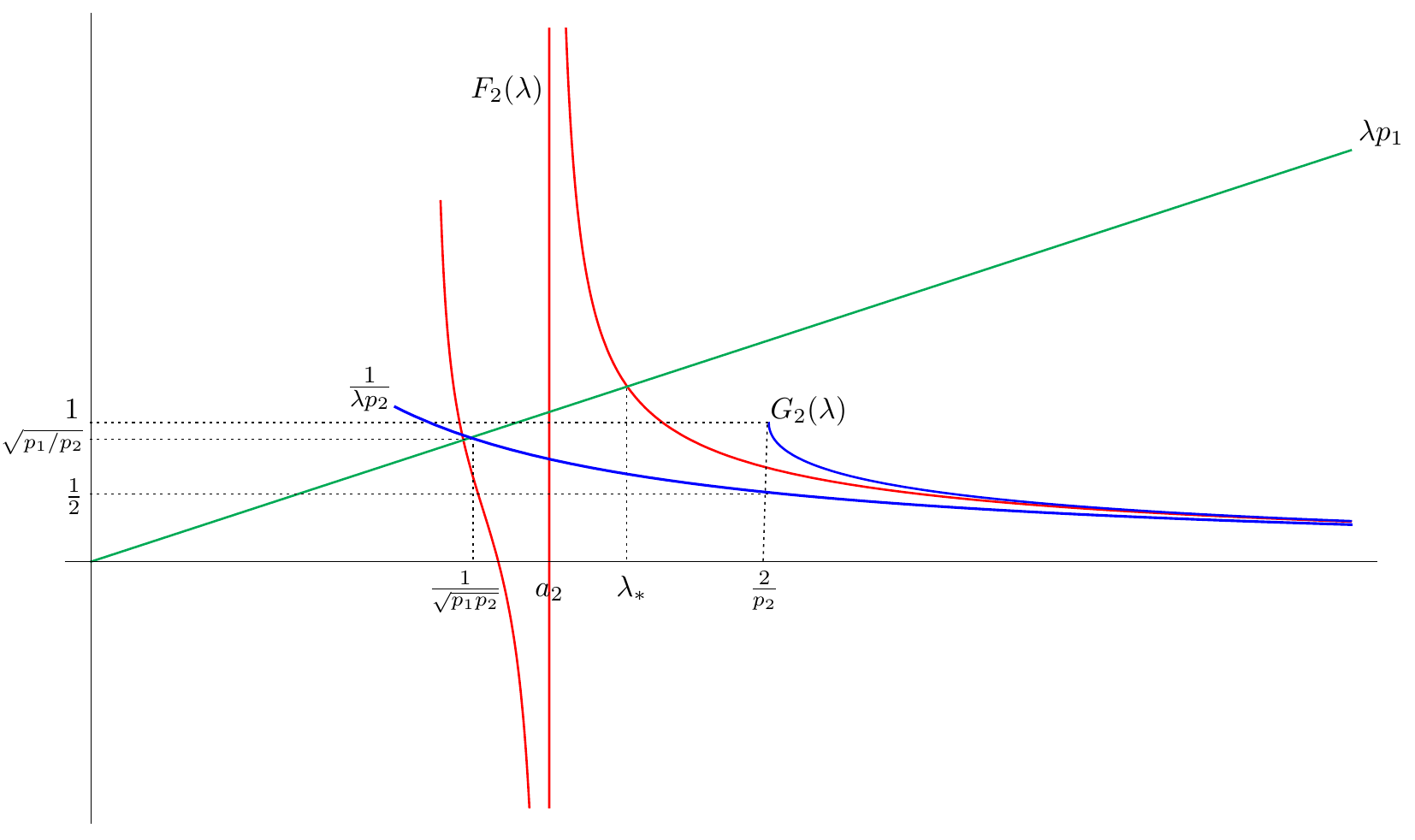}
  \caption{Solving for $\lambda_*$ when $1/\sqrt{p_{1} p_{2}} < 2/p_{2}$, and when $F_2$ has vertical asymptotes larger than $\lambda_0$. The above plot was obtained by numerically computing $F_2$, where the coefficients $p_n$ arise in the study of the SIPM equation.  In the numerical computation we have used the explicit formula for $p_n$ given in \eqref{pn}, and set all $a=b=1$, and $\beta = 1.5$.}
  \label{fig:SIMP}
\end{figure}

In summary, we have proven the existence of a solution $\lambda_*$ to \eqref{eq:*}, which satisfies the bound \eqref{eq:lambda*:bounds}.
Recalling that $F_n(\lambda_*) = -\eta_n$, as defined by \eqref{eq:eta:rec}, we obtain that the $c_n$'s defined by \eqref{eq:cn:def} satisfy
\begin{align*}
 | c_n |  =   p_n F_n(\lambda_*) \ldots F_2(\lambda_*).
\end{align*}
In order to estimate the $|c_n|$, we recall that one can find $n_0 \geq 2$ such that
\begin{align*}
p_{n_{0}} < 4 \sqrt{p_{1} p_{2}} \leq p_{n_{0}+1},
\end{align*} or equivalently $2/p_{n_0+1} \leq \lambda_0/2 < 2/p_{n_{0}}$. For this fixed $n_{0}$, using that $\lambda_{*} \in A_{j}$ for all $2 \leq j \leq n_{0}$, by \eqref{eq:Aj} we can estimate
\begin{align}
F_{n_0}(\lambda_*) \ldots F_2(\lambda_*) < \lambda_{*}^{n_{0}-1} p_{n_{0}-1} \ldots p_{2} \leq \left( \frac{4 \lambda_{*}}{\lambda_0} \right)^{n_0-1} \label{eq:F:less:n0}
\end{align}
so that we have
\begin{align}
|c_{n}| \leq 4^{n_0} \lambda_*^{n_0-1} \lambda_0^{-n_0} \leq 4^{n_0} p_2 (\lambda_* \sqrt{p_1 p_2})^{n_0-1} \leq p_2 2^{3 n_0 -1}
\label{eq:cn:SIMP:low}
\end{align}
for $n\leq n_0$, where we have also used the sharp bound $\lambda_* \leq \max\{ 2/p_2, 1/p_1\}$, and $p_2 < 4 p_1$.
Then, using that for $n \geq n_0+1$ we may use the bound \eqref{eq:Fn:Gn:bound}, from the monotonicity of the $p_n$s, and the choice of $n_0$ we obtain from \eqref{eq:F:less:n0} that
\begin{align}
|c_n| &= p_n F_{n}(\lambda_*) F_{n-1}(\lambda_*) \ldots F_{n_0+1}(\lambda_*) (4 \lambda_{*})^{n_{0}-1} \lambda_0^{1-n_0} \notag\\
& \leq p_n \frac{ 2^{n-n_0}}{\lambda_*^{n-n_0}  p_{n} \ldots p_{n_0+1}} (4 \lambda_{*})^{n_{0}-1} \lambda_0^{1-n_0} \notag\\
& \leq \frac{ 4^{n_0-1} 2^{n-n_0}}{\lambda_*^{n+1-2 n_0}  \lambda_0^{n_0-1} p_{n_0+1}^{n-n_0-1}} \leq \frac{2^{n_0}}
{\lambda_0} \left( \frac{\lambda_0}{2 \lambda_*} \right)^{n-2 n_0}  \leq \frac{2^{3 n_0 - n}}{\lambda_0} \leq p_2 2^{3n_0 -n}
\label{eq:cn:SIMP:high}
\end{align}
by recalling that by construction we have $\lambda_0 < \lambda_*$.  We obtain that the $c_n$'s decay exponentially for all $n\geq 3 n_0$. Moreover,  from \eqref{eq:cn:SIMP:low} and \eqref{eq:cn:SIMP:high} we obtain
\begin{align*}
\| n^s c_n\|_{\ell^2(\NN)}^2 \leq (3 n_0)^{2s} \sum_{n=0}^{3n_0} c_n^2 + p_2^2 \sum_{k\geq 0} 2^{-2k} \leq C n_0^{2s} \| c_n\|_{\ell^2(\NN)}^2 +2 p_2^2  \leq C  n_0^{2s} \| c_n\|_{\ell^2(\NN)}^2,
\end{align*}
which proves \eqref{eq:Sobolev:Estimate}, and concludes the proof of the theorem. Here we also used $p_2 = c_2/(p_1 \lambda_*) \leq 2 c_2$.
\end{proof}

\subsection*{Acknowledgements} SF was supported in part by the NSF grant DMS-0803268. FG was supported in part by the grant MTM2008-03754 of the MCINN(Spain), the grant StG-203138CDSIF of the ERC, and the NSF grant DMS-0901810. VV was supported in part by and AMS-Simmons travel award.

\end{document}